\documentclass[10pt,authoryear]{article}%
\usepackage[utf8]{inputenc}

\usepackage{amsmath}
\usepackage{enumitem}
\usepackage{amsfonts}
\usepackage{mathrsfs}
\usepackage{amssymb,bbm, color}
\usepackage[hidelinks,colorlinks=true,linkcolor=blue,citecolor=blue]{hyperref}
\usepackage{graphicx}
\numberwithin{equation}{section}
\usepackage[body={15.5cm,21cm}, top=3cm]{geometry}%
\providecommand{\U}[1]{\protect\rule{.1in}{.1in}}
\providecommand{\U}[1]{\protect \rule{.1in}{.1in}}
\newtheorem{theorem}{Theorem}[section]

\newtheorem{definition}[theorem]{Definition}

\newtheorem{lemma}[theorem]{Lemma}

\newtheorem{proposition}[theorem]{Proposition}
\newtheorem{remark}[theorem]{Remark}

\newtheorem{assumption}[theorem]{Assumption}
\newenvironment{proof}[1][Proof]{\noindent \textbf{#1.} }{\  \rule{0.5em}{0.5em}}

\def\mR{{\mathbb R}}

\definecolor{purple}{rgb}{0.5,0,1}

\usepackage{natbib}
\begin{document}
	\title{Well-posedness and Propagation of Chaos for McKean--Vlasov Stochastic Variational Inequalities}
	\author{Ning Ning\thanks{Department of Statistics,
			Texas A\&M University, College Station, Texas, USA}
		\and Jing Wu\thanks{School of Mathematics, Sun Yat-sen University,
			Guangzhou, Guangdong 510275, China. Corresponding author.}}
	\date{}
	\maketitle
	
	\begin{abstract}
In this paper, we study a broad class of McKean–Vlasov stochastic variational inequalities (MVSVIs), where both the drift coefficient $b$ and the diffusion coefficient $\sigma$ depend on  time $t$, the state $X_t$ and its distribution $\mu_t$. We establish the strong well-posedness, when $b$ is superlinear growth and locally Lipschitz continuous, and  $\sigma$ is locally H\"older continuous, both with respect to $X_t$ and  $\mu_t$. Additionally, we present the first propagation of chaos result for MVSVIs.
	\end{abstract}
	
	
	\textbf{MSC-classification}: 60H10; 49J53; 60K35.

	
	
	\allowdisplaybreaks
	\newcommand{\ABS}[1]{\Bigg(#1\Bigg)} 
	\newcommand{\veps}{\varepsilon} 
	
\section{Introduction}	
We firstly give the background in Section \ref{sec:intro_background} and then state our contributions in Section \ref{sec:intro_contributions}, followed by the organization of the paper in Section \ref{sec:organization}.

\subsection{Background}
\label{sec:intro_background}
Variational inequality (VI) modeling, analysis, and computation hold significant importance across various applications, including economics, finance, optimization, and game theory  \citep{robinson1979generalized, robinson1982generalized}, and hence the literature on VIs is extensive. Due to page constraints, we refer the classical work  \citep{rockafellar2009variational} and the references therein for a comprehensive exploration and understanding of VIs. While the initial research of VIs primarily unfolded in a deterministic setting without accounting for uncertainty in a problem's data, recent years have witnessed a trajectory towards incorporating stochastic elements \citep{rockafellar2017stochastic}. 
In this paper, we consider stochastic VIs (SVIs), whose applications encompass constrained time-dependent physical systems with unilateral constraints, differential Nash games, and hybrid engineering systems with variable structures \citep{pang2008differential}. 
For example,  \cite{kree1982diffusion} introduced this type of stochastic differential equations (SDEs) to model random oscillations of the antiseismic design in engineering. \cite{cepa1997diffusing,cepa1998problame, cepa2001brownian} showed that particle systems with possible collisions can be described by SVIs.  \cite{bensoussan2006stochastic} and \cite{bensoussan2012long} applied SVIs to model the behavior of elasto-plastic oscillator excited by white noises and related problems. SVIs are also applied in   \cite{shanbhag2013stochastic} and \cite{rockafellar2017stochastic} 
for optimization and  equilibrium problems involving random data.

Our focus lies specifically on the SVI manifesting as a SDE with a subdifferential operator associated with a proper, convex, and lower-semicontinuous function $\psi(x)$. 
Notably, when $\psi(x)$ takes the specific form of an indicator function of a closed convex domain with a nonempty interior, the SVI is known as a reflected SDE (RSDE). That is, for a convex domain $\overline{D}\subset \mathbb{R}^d$,
\begin{align*}\psi(x)=\mathbf{1}_{\overline{D}}(x)
	=\left\{ \begin{array}{lcl}
		0 & \mbox{for}
		& x\in \overline{D}\\ 
		+\infty \;\;  & \mbox{for} & x\in \mathbb{R}^d\backslash \overline{D}
	\end{array}\right.
\end{align*}
and the subdifferential of $\psi$ is 
\begin{align*}
	\partial\psi(x)=\partial\mathbf{1}_{\overline{D}}(x)
	=\left\{ \begin{array}{lcl}
		0 & \mbox{for}
		& x\in D\\ 
		\Pi \;\;  & \mbox{for}
		& x\in \partial D\\ 
		\emptyset & \mbox{for} & x\in \mathbb{R}^d\backslash \overline{D}
	\end{array}\right.,
\end{align*}
where  $\partial$ denotes the subdifferential operator and  $\Pi=\big\{y \in \mathbb{R}^d :\, \langle y,x-z \rangle\geq 0,\, \forall z\in \overline{D}\big\}$. We refer interested readers to \cite{rockafellar1969convex,rockafellar1970maximal} for further understanding of the properties of the  subdifferential mappings.

RSDEs have found extensive applications and garnered considerable interest in the probability community, such as \cite{ren2013optimal} investigating RSDEs with jumps and their associated optimal control problems, \cite{hu2015parameter} estimating parameters for RSDEs with discrete observations, \cite{han2016optimal}
optimally pricing barriers in a regulated market using RSDEs, \cite{ren2016approximate} focusing on the approximation continuity and support theory of general RSDEs, and \cite{wu2018limit} exploring limit theorems and support properties of RSDEs with oblique reflection on nonsmooth domains, \cite{ren2020equivalence} discussing the equivalence of viscosity and distribution solutions for second-order partial differential equations (PDEs) associated with oblique RSDEs.
Advancements in SVIs span various topics, encompassing the well-posedness of SVIs initially established by \cite{bensoussan1994parabolic} in finite-dimensional spaces and by \cite{bensoussan1997stochastic} in infinite-dimensional spaces. \cite{asiminoaei1997approximation} introduced a simplified approach to investigate SVIs, while \cite{zualinescu2002second} explored a class of Hamilton-Jacobi-Bellman inequalities through the lens of an optimal stochastic control problem with an SVI as the state equation. \cite{ren2012regularity} analyzed the regularity of invariant measures for multivalued SDEs, \cite{ning2021well} established the well-posedness and conducted stability analysis of correlated two multi-dimensional path-dependent SVIs, and \cite{ning2023multi} investigated multi-dimensional path-dependent forward-backward SVIs.

SDEs whose coefficients depending on the distributions of the solutions are commonly  referred to as McKean-Vlasov SDEs (MVSDEs). They find applications in diverse fields such as statistical physics, mean-field games in large-scale social interactions and so on \citep{carmona2018probabilistic}. Progress in the study of MVSDEs includes the following: \cite{carmona2015forward} provided a detailed probabilistic analysis of controlled MVSDEs;
in the case of globally Lipschitz continuous diffusion coefficient and one-sided globally Lipschitz continuous drift  coefficient, \cite{wang2018distribution} established the well-posedness of strong solutions using an iteration approximation; \cite{crisan2018smoothing} investigated the regularity of the solutions of MVSDEs using Malliavin calculus; under super-linear growth conditions, \cite{dos2019freidlin} utilized the fixed point theorem to prove the well-posedness of strong solutions; \cite{huang2019distribution} proved the well-posedness of MVSDEs with non-degenerate diffusion under integrable conditions;
under weaker integrability assumptions, \cite{rockner2021well} obtained strong well-posedness for MVSDEs with constant diffusion coefficient under the total variation distance; \cite{li2023strong} proved the strong convergence of tamed Euler-Maruyama approximation for MVSDEs with super-linear growth drift
and globally H\"older continuous diffusion coefficients. For further research on MVSDEs with non-Lipschitz continuous coefficients, we refer to works such as \cite{bao2021approximations, de2020strong, hammersley2021McKean, li2023strong}.

McKean-Vlasov SVI (MVSVI), as the name indicates, is the SVI of McKean-Vlasov type, where the coefficients depend on the law of the solution.
MVSVIs generalize MVSDEs with reflection, which are denoted as RMVSDEs, of convex domains. \cite{sznitman1984nonlinear} was the first to prove the well-posedness of RMVSDEs in smooth bounded domains.  Strong restrictions on the coefficients of RMVSDEs as being Lipschitz continuous and bounded, are usually imposed for the sake of simplicity in proofs. However, that excludes a broad class of models, even including the very basic Cox–Ingersoll–Ross (CIR) model which is characterized by H\"older continuous coefficients.
Recently, several notable contributions have expanded our understanding of RMVSDEs: \cite{wang2021distribution,wang2023exponential} explored the well-posedness and functional inequalities for RMVSDEs with singular or monotone coefficients; \cite{adams2022large} proved the well-posedness of RMVSDEs in general convex domains with super-linear growth components in both space and measure;  \cite{huang2022singular} established the well-posedness of singular MVSDEs, where the drift contains a term growing linearly in space and distribution and a locally integrable term independent of distribution, while the noise coefficient is weakly differentiable in space and globally Lipschitz continuous in distribution with respect to (w.r.t.) the sum of Wasserstein and weighted variation distances.

The solutions of MVSDEs, also known as nonlinear diffusions, provide a probabilistic representation of a class of nonlinear PDEs. An important example of such nonlinear PDEs is the study in \cite{McKean1966class}, which describes the limiting behavior of an individual particle within a large system of particles undergoing diffusive motion and interacting in a ``mean-field'' sense as the population size grows to infinity. As the system approaches the limit, any finite subset of particles becomes asymptotically independent of each other. This phenomenon, known as the propagation of chaos (POC), has been extensively studied by various authors (see, e.g. \cite{sznitman1991topics,graham1996asymptotic,antonelli2002rate}).
In the context of MVSDEs, \cite{gartner1988mckean} proved the well-posedness of a weak solution and the POC. \cite{Lacker2018on} obtained a strong form of POC and some modest well-posedness results. \cite{andreis2018mckean} established strong well-posedness and POC for MVSDEs with jumps under globally Lipschitz continuous assumptions on the coefficients, which was further improved in \cite{erny2022well} under locally Lipschitz continuity assumptions.

\subsection{Our contributions}
\label{sec:intro_contributions}
In this paper, our contributions are establishing the strong well-posedness under the weakest known regularity condition for a broad class of MVSVIs, and proving its POC. Our results and proof strategies as summarized below.

\subsubsection{Well-posedness of MVSVI}
\label{sec:into_Well-posedness}
We investigate the following time-inhomogeneous MVSVI:
\begin{align}\label{eqn:mvsvi}
	X_t \in X_0+\int_0^tb(s, X_s, \mu_{s})ds+\int_0^t\sigma(s, X_s,\mu_{s})dB_s-\int_0^t\partial\psi(X_s)ds,
\end{align}
where $t\in[0, T]$ with $T>0$, and $\mu\in \mathcal{P}(\mathbb{R})$, i.e., the distribution of the random variable $X$ in the space of probability measures on $\mathbb{R}$.
We consider the drift coefficient $b:[0, T]\times\mathbb{R}\times\mathcal{P}(\mathbb{R})\rightarrow\mathbb{R}$ and the diffusion coefficient $\sigma: [0, T]\times\mathbb{R}\times\mathcal{P}(\mathbb{R})\rightarrow\mathbb{R}$ as measurable stochastic functions.  Here, $\psi: \mathbb{R}\rightarrow\mathbb{R}$ is a convex function and
\begin{align}
	\label{eqn:subdifferential}
	\partial\psi(x):=\Big\{z\in \mathbb{R}: (x'-x)z\le\psi(x')-\psi(x),\;\forall x'\in\mathbb{R} \Big\}
\end{align}
with domain $D(\partial\psi):=\{x\in\mathbb{R};\partial\psi(x)\neq\emptyset\}$; see Theorem \ref{thm:Properties_of_subdifferential}  for its properties. We consider stochastic initial conditions and work on a filtered probability space $(\Omega, \mathscr{F}, \mathbb{F}=\{\mathscr{F}_t\}_{t\in [0,T]}, \mathbb{P})$ which supports an $\mathbb{F}$-adapted standard Brownian motion $B$.

The definition of a strong solution to MVSVI \eqref{eqn:mvsvi} is provided in Definition \ref{def:mvsvi}.  In this paper, we establish the strong well-posedness in Theorem \ref{thm:mvsvi_wellposedness}, under Assumption \ref{assumption2} where we merely suppose super-linear growth  and locally Lipschitz continuous of $b$ and locally H\"older continuous of $\sigma$ both w.r.t. the state and the distribution. To our best knowledge, this is the first time that the well-posedness  of SVI (resp. MVSVI)  is established under such weak regularity conditions. Our MVSVI \eqref{eqn:mvsvi} covers the following general McKean-Vlasov SDE:
\begin{align*}
	X_t = X_0+\int_0^tb(s, X_s, \mu_{s})ds+\int_0^t\sigma(s, X_s,\mu_{s})dB_s,
\end{align*}
whose strong convergence of the Euler-Maruyama schemes is established in \cite{li2023strong} under globally H\"older continuous diffusion coefficients.  Our MVSVI \eqref{eqn:mvsvi}  further covers the following general SDE:
\begin{align*}
	X_t= X_0+\int_0^tb(s, X_s)ds+\int_0^t\sigma(s, X_s)dB_s,
\end{align*}
whose strong convergence of the tamed Euler-Maruyama approximation is proved in  \cite{ngo2019tamed} with the same conditions on the coefficients as ours.

Clearly, our MVSVI \eqref{eqn:mvsvi} covers the following general SVI (which includes general RSDEs):
\begin{align}
	\label{eqn:svi}
	X_t \in X_0+\int_0^tb(s, X_s)ds+\int_0^t\sigma(s, X_s)dB_s-\int_0^t\partial\psi(X_s)ds.
\end{align} 
To date, globally H\"older continuous coefficients are needed  \citep{ning2021well}. Hence, to achieve the strong well-posedness of the MVSVI  \eqref{eqn:mvsvi} with locally H\"older continuous coefficients, we first need to establish that of SVI \eqref{eqn:svi}, whose definition of a strong solution is provided in Definition \ref{ecd01}. That goal is achieved in Theorem \ref{thm:svi_wellposedness}, under Assumption \ref{assumption1} where we merely suppose locally Lipschitz continuous and super-linear growth of $b$ and locally H\"older continuous of $\sigma$ both w.r.t. the state. 
To prove Theorem \ref{thm:svi_wellposedness}, we employ the Yosida-Moreau approximation function \citep{barbu2010nonlinear} defined in equation \eqref{eqn:Yosida-Moreau approximation}. During the procedure, we utilize not only the standard It\^o's formula but also the generalized It\^o's formula (Theorem \ref{tb8}) to bound the approximation $\psi^{n}\left(X_{t}^{n}\right)$. This is necessary because $\psi^n(\cdot)$ lacks a continuous second derivative, rendering the standard It\^o's formula inapplicable. However, since $\psi^n(\cdot)$ possesses a second derivative almost everywhere, we can apply the generalized It\^o's formula.
Furthermore, we incorporate the Yamada-Watanabe function defined in equation \eqref{eqn:Yamada-Watanabe_func} to facilitate our analysis. The properties of this function are outlined in Theorem \ref{thm:Properties_of_YW}. Finally, employing  properties of the Yosida-Moreau function (Theorem \ref{thm:Yosida-Moreau}) and properties of the integration of functions of bounded variation (Theorem \ref{dd16}), by means of the truncation arguments, we establish the well-posedness as desired.

Techniques applicable to SDEs usually cannot be directly applied on MVSDEs. Not surprisingly, to prove Theorem \ref{thm:mvsvi_wellposedness}, we have to abandon the Yosida-Moreau approximation that was used to prove the existence of a strong solution for the SVI \eqref{eqn:svi}, and instead employ the Picard iteration method which is also known as the Banach fixed point theorem. Specifically, we establish that a subsequence, each of whose strong well-posedness is covered in Theorem \ref{thm:svi_wellposedness}, converges in distribution to a limit that is proven to be a solution to the equation. Therefore, we need to control the variation between two consecutive steps of the scheme, and rigorously demonstrate that the Picard scheme is a solution to the equation even though the subsequence convergence is only in distribution. We utilize the concept of weak convergence in $\mathcal{P}_1(\mathbb{R})$ in the Wasserstein sense, as defined in Definition \ref{def:Wasserstein_cvg} and described in Theorem \ref{thm:Wasserstein_cvg}, and employ Skorohod's representation theorem.
The Yamada-Watanabe function continues to be employed to handle the locally  H\"older continuity of the diffusion coefficient.
To establish uniqueness, we apply It\^o's formula to the Yamada-Watanabe function of the difference between two solutions. We introduce an auxiliary quantity $\Lambda{(t)}$ in equation \eqref{eqn:def_Lambda} and reduce the problem to proving $\sup_{0\leq t\leq T} \Lambda(t)=0$. To accomplish this, we utilize the method of contradiction with the aid of Osgood's lemma, which can be seen as a generalization of Gr\"onwall's lemma for the case of local coefficients.
To the best of our knowledge, this is the first time that Osgood's lemma is used in the context of MVSVI.

\subsubsection{Propagation of chaos}
Consider the following $N$-particle SVI system: for $1\leq i\leq N$ with $N\in \mathbb{N}$,
\begin{align}
	\label{eqn:eqn_poc}
	d X_{t}^{N, i} \in b(X_{t}^{N, i}, \mu_{t}^{N}) d t+\sigma(X_{t}^{N, i}, \mu_{t}^{N}) d B_{t}^{i}-\partial \psi(X_{t}^{N, i}) d t,
\end{align}
where $\mu^{N}(:=\frac{1}{N} \sum_{i=1}^{N} \delta_{X^{N, i}})$ is the empirical distribution,
$\{X_{0}^{N, i}\}_{1\leq i\leq N}$ are independent and identically distributed (i.i.d.), and $\{B^i\}_{1\leq i\leq N}$ are independent standard Brownian motions.
We first show that its strong well-posedness can be established in Theorem
\ref{thm:wellposedness_POC} under Assumption \ref{assumption2}. Then consider the McKean-Vlasov limit system
\begin{align}
	\label{eqn:MV_limit}
	d \overline{X}_{t}^{i} \in b(\overline{X}_{t}^{i}, \overline{\mu}_{t}) d t+\sigma(\overline{X}_{t}^{i}, \overline{\mu}_{t}) d B_{t}^{i}-\partial \psi(\overline{X}_{t}^{i}) d t,
\end{align}
where $\overline{\mu}_t$ is the law of  $\overline{X}_t=\{\overline{X}_t^{i}\}_{i\geq 1}$. The POC is proved in Theorem \ref{thm:POC} in the form that $\mathbb{E}\sup _{t \leq T}\big|X_{t}^{N, i}-\overline{X}_{t}^{i}\big| \rightarrow 0$ as $n$ goes to infinity,
for all $T>0$.
The techniques used are those in establishing the well-posedness of MVSVI \eqref{eqn:mvsvi}. We expect our proof strategies and techniques will be useful in other MVSVI analyses.

\subsection{Organization of the paper}
\label{sec:organization}

The rest of the paper proceeds as follows. 		In Section \ref{appendix_Properties}, we provide existing results on important quantities, which will be the workhorse of our proofs. 
In Section \ref{sec:Well-posedness},
we prove the strong well-posedness of the SVI  \eqref{eqn:svi} and the MVSVI \eqref{eqn:mvsvi}. In Section \ref{sec:POC},
we prove the POC of the $N$-particle system \eqref{eqn:eqn_poc}.
Throughout the paper, the letter $C$, with or without subscripts, will denote a positive constant whose value may change for different usage. Thus, $C+C = C$ and $CC = C$ are understood in an appropriate sense. Similarly, $C_{\alpha}$ denotes the generic positive constant depending on parameter $\alpha$.

\section{Classical results}
\label{appendix_Properties}
To make the paper self-contained, we summarize here some classical results. The following theorem covers properties of subdifferential operator.


\begin{theorem}[\cite{rockafellar1970maximal}]
	\label{thm:Properties_of_subdifferential}	
	The subdifferential operator $\partial\psi(x)$ is monotone, that is, for any $x,x'\in \mathbb{R}^n$, $z\in \psi(x), z'\in \psi(x')$, we have
	$$\langle x-x',z-z'\rangle\ge 0.$$
	The subdifferential operator is also maximally monotone, that is, if $x,z\in\mathbb{R}^n$ satisfy that	
	$$\langle x-x',z-z'\rangle\ge0,\qquad \forall x',z'\in \partial\psi(x'),$$
	then $z\in\partial\psi(x)$.
\end{theorem}	

The following theorem is the generalized It\^o's formula.
\begin{theorem}[\cite{karatzas1991brownian}]
	\label{tb8}
	If $f:\mathbb{R}\rightarrow\mathbb{R}$ has an absolutely continuous first derivative on $\mathbb{R}$ and $X_t=X_0+M_t+B_t$ is a continuous semimartingale of the form
	\begin{equation*}
		X_t=X_0+M_t+B_t,
	\end{equation*}
	then the formula still holds $\mathbb{P}$-a.s.
	$$f(X_t)=f(X_0)+\int_0^tf'(X_s)dM_s+\int_0^tf'(X_s)dB_s+\frac{1}{2}\int_0^tf''(X_s)d[M]_s,$$
	where $[M]_t$ denotes the quadratic variation process of the local martingale $M_t$.
\end{theorem}

The following theorem covers properties of Yosida-Moreau approximation function.
\begin{theorem}[\cite{barbu2010nonlinear}]
	\label{thm:Yosida-Moreau}
	The Yosida-Moreau approximation function
	$\psi^n$ is defined as
	\begin{align}
		\label{eqn:Yosida-Moreau approximation}
		\psi^n(x):=\inf\left\{\frac{n}{2}|x'-x|^2+\psi(x');\, x'\in \mathbb{R}\right\},
	\end{align}
	and $J_n$ is defined as
	\begin{align}
		\label{eqn:Jnx}
		J_nx:=x-\frac{1}{n}\nabla\psi^n(x).
	\end{align}
	Then for every $n$, $\psi^n$ is convex and differentiable. Moreover,  for any $x,y\in \mathbb{R}$,
	$$
	\left \{\begin{aligned}
		&(x-y)(\nabla\psi^n(x)-\nabla\psi^m(y))\ge -\left(\frac{1}{n}+\frac{1}{m}\right)\nabla\psi^n(x)\nabla\psi^m(y),\\
		&|\nabla\psi^n(x)-\nabla\psi^n(y)|\le n|x-y|,\\
		&(x-y)\nabla\psi^n(x)\geq\psi^n(x)-\psi^n(y)\geq\psi(J_nx)-\psi(y),\\
		&\psi^n(x)=\psi(J_nx)+\frac{1}{2n}|\nabla\psi^n(x)|^2,\\
		&\psi(J_nx)\le\psi^n(x)\le\psi(x),\\
		&|J_nx-J_ny|\le|x-y|,\\
		&\lim\limits_{n\rightarrow \infty}J_nx=P_{\overline{D}}(x).
	\end{aligned}\right.
	$$
	Here, $P_{\overline{D}}(x)$ denotes the projection of $x$ onto $\overline{D}$ where $D:=D(\partial\psi)$.
\end{theorem}	

When $x> 0$, let $\varphi_{\epsilon,\delta}(x)$ be a continuous function supported on $[\frac{\epsilon}{\delta},\epsilon]$ for any $\delta>1$ and $\epsilon\in (0,1)$, satisfying that
$$
0\le \varphi_{\epsilon,\delta}(x) \le \frac{2}{|x|\ln(\delta)}\quad\text{and}\quad\int_{\epsilon/ \delta}^{\epsilon}\varphi_{\epsilon,\delta}(x)dx=1.
$$
Define the Yamada-Watanabe function (\cite{yamada1971uniqueness})
\begin{align}
	\label{eqn:Yamada-Watanabe_func}
	V_{\epsilon,\delta}(x):=\int_0^x\int_0^y\varphi_{\epsilon,\delta}(z)dzdy.
\end{align}
The following theorem covers properties of the Yamada-Watanabe function.
\begin{theorem}[\cite{yamada1971uniqueness}]
	\label{thm:Properties_of_YW}
	The Yamada-Watanabe function $V_{\epsilon,\delta}(x)$ satisfies that
	\begin{align*}
		|x| - \epsilon \le V_{\epsilon,\delta}(x) \le|x|, \qquad
		0 \le \operatorname{sgn}(x)V_{\epsilon,\delta}'(x) \le1,\\
		\text{and}\quad 0 \le V_{\epsilon,\delta}^{''}(x) \le \frac{2}{|x|\ln(\delta)}\mathbbm{1}_{[\epsilon/\delta,\epsilon]}(|x|).
	\end{align*}
\end{theorem}

The following theorem covers properties of the integration of functions of bounded variation.
\begin{theorem}[\cite{cepa1998problame}]
	\label{dd16}
	Let $k^m: [0, T] \rightarrow \mathbb{R}^n$ be a sequence of continuous functions of bounded variation such that the total variation on $[0,T]$ (denoted by $ |k^m|_0^T$) satisfies
	$\sup_{m} |k^m|_0^T < \infty$ and
	$$\lim_{m \rightarrow \infty} \sup_{0 \leq s \leq T} |k^m_s - k_s| = 0. $$ Then $k$ is also a function of bounded variation, and for any sequence of continuous functions $f^m$ such that
	$$\lim_{m \rightarrow \infty} \sup_{0 \leq s \leq T} |f^m_s - f_s| = 0,$$ we have that for any $0\le s<t\le T$,
	$$\int_s^t f^m_rdk^m_r\rightarrow \int_s^tf_rdk_r.$$
\end{theorem}

The following theorem is Osgood's lemma.
\begin{theorem}[\cite{bahouri2011fourier}]
	\label{thm:Osgood}
	Let $f$ be a measurable function from $[t_0, T ]$ to $[0, a]$, $\gamma$ be an 
	integrable function from $[t_0,T]$ to $\mathbb{R}^+$, and $g$ be a continuous, nondecreasing function from $[0, a]$ to $\mathbb{R}^+$. Assume that, for some nonnegative real number $c$, the function $f$ satisfies
	$$f(t)\leq c+\int_{t_0}^{t} \gamma(s)g(f(s))ds \quad \text{for a.e. } t\in [t_0,T].$$
	If $c$ is positive, then we have, for a.e. $t\in [t_0, T ]$,
	$$-\mathcal{M}(f(t))+\mathcal{M}(c)\leq \int_{t_0}^{t} \gamma(s)ds \quad \text{with}\quad \mathcal{M}(x)=\int_{x}^{a}\frac{dr}{g(r)}.$$
	If $c=0$ and $\int_0^a\frac{dr}{g(r)}dr=\infty$, then $f=0$ a.e..
\end{theorem}

The following definition and theorem are taken from Definition $6.8$  and Theorem $6.9$ in \cite{villani2009optimal}, respectively, regarding the convergence in the Wasserstein sense. The notation $\mu_k\rightarrow \mu $ means that $\mu_k$ converges weakly to $\mu $.
\begin{definition} [Weak convergence in $\mathcal{P}_p$]
	\label{def:Wasserstein_cvg}
	Let $(\mathcal{X}, d)$ be a Polish space, and $p \in[1, \infty)$. Let $\left(\mu_k\right)_{k \in \mathbb{N}}$ be a sequence of probability measures in $\mathcal{P}_p(\chi)$ and let $\mu$ be another element of $\mathcal{P}_p(\mathcal{X})$. Then $\left(\mu_k\right)$ is said to converge weakly in $\mathcal{P}_p(\mathcal{X})$ if any one of the following equivalent properties is satisfied for some (and then any) $x_0 \in \mathcal{X}$:
	\begin{enumerate}
		\setlength\itemsep{0.3em}
		\item $\mu_k \rightarrow \mu$ and $\int d\left(x_0, x\right)^p d \mu_k(x) \longrightarrow \int d\left(x_0, x\right)^p d \mu(x)$;
		\item $\mu_k \rightarrow \mu$ and $\limsup _{k \rightarrow \infty} \int d\left(x_0, x\right)^p d \mu_k(x) \leq \int d\left(x_0, x\right)^p d \mu(x)$;
		\item  $\mu_k \rightarrow \mu$ and $\lim _{R \rightarrow \infty} \limsup _{k \rightarrow \infty} \int_{d\left(x_0, x\right) \geq R} d\left(x_0, x\right)^p d \mu_k(x)=0$;
		\item  For all continuous functions $\varphi$ with $|\varphi(x)| \leq C\left(1+d\left(x_0, x\right)^p\right)$ for some constant $C>0$, one has
		$$\int_\chi \varphi(x) d \mu_k(x) \longrightarrow \int_\chi \varphi(x) d \mu(x).$$
		
	\end{enumerate}
	
	%
	%
\end{definition}

\begin{theorem}[\cite{villani2009optimal}]
	\label{thm:Wasserstein_cvg}
	Let $(\mathcal{X}, d)$ be a Polish space, and $p \in[1, \infty)$; then the Wasserstein distance $W_p$ metrizes the weak convergence in $\mathcal{P}_p(\mathcal{X})$. In other words, if $\left(\mu_k\right)_{k \in \mathbb{N}}$ is a sequence of measures in $\mathcal{P}_p(\mathcal{X})$ and $\mu$ is another measure in $\mathcal{P}_p(\mathcal{X})$, then the statements
	$$\mu_k \text{ converges weakly in } \mathcal{P}_p(\mathcal{X}) \text{ to } \mu$$
	and
	$$W_p\left(\mu_k, \mu\right) \rightarrow 0$$
	are equivalent.
\end{theorem}
The following lemma is Lemma 3.1 of \cite{erny2022well}.
\begin{lemma}[\cite{erny2022well}]
	\label{thm:erny}
	Let $N \in \mathbb{N}^*$ the set of positive integers, $T>0$, and $\left(x^k\right)_{1 \leq k \leq N}$ and $\left(x_n^k\right)_{1 \leq k \leq N}(n \in \mathbb{N})$ be càdlàg functions. Define
	$$
	\mu_n(t):=N^{-1} \sum_{k=1}^N \delta_{x_n^k(t)} \quad\text { and } \quad\mu(t):=\frac1N\sum_{k=1}^N \delta_{x^k(t)} .
	$$
	Let $\lambda_n(n \in \mathbb{N})$ be continuous, increasing functions satisfying $\lambda_n(0)=0, \lambda_n(T)=T$, and that, for any $1 \leq k \leq N$, as $n\to\infty$,
	$$
	\sup _{0 \leq t \leq T}\left|x_n^k(t)-x^k\left(\lambda_n(t)\right)\right|\to0, \quad\text { and } \quad\sup _{0 \leq t \leq T}\left|t-\lambda_n(t)\right|\to0.
	$$
	Then,
	$$
	\sup _{0 \leq t \leq T} W_1\big(\mu_n(t),\, \mu\left(\lambda_n(t)\right)\big) \underset{n \rightarrow \infty}{\longrightarrow} 0 .
	$$
\end{lemma}

\section{Well-posedness}
\label{sec:Well-posedness}
In this section, we first establish the strong well-posedness of the SVI  \eqref{eqn:svi} in Section \ref{sec:Well-posedness_SVI} and then that of the MVSVI  \eqref{eqn:mvsvi} in Section \ref{sec:Well-posedness_MVSVI}.

\subsection{Well-posedness of the SVI}
\label{sec:Well-posedness_SVI}

We first give the definition of solutions  to equation \eqref{eqn:svi} recalled here as follows:
\begin{align*}
	X_t \in X_0+\int_0^tb(s, X_s)ds+\int_0^t\sigma(s, X_s)dB_s-\int_0^t\partial\psi(X_s)ds,
\end{align*} 
where $b:[0, T]\times\mathbb{R} \rightarrow\mathbb{R}$ and $\sigma: [0, T]\times\mathbb{R} \rightarrow\mathbb{R}$ as measurable stochastic functions.
\begin{definition}\label{ecd01}
	A pair of continuous adapted processes $(X, \phi)$ defined on $(\Omega, \mathscr{F}, \{\mathscr{F}_t\}_{t\in [0,T]}, \mathbb{P})$ is called a strong solution to equation \eqref{eqn:svi} if it satisfies the following conditions:
	\begin{enumerate}
		\setlength\itemsep{0.15em}
		\item For any $t\in[0, T]$, $X_t\in\overline{D(\partial\psi)}$ a.s..
		\item For any $t\in[0, T]$,
		$$\int_0^t\mathbb{E}|b(s, X_s)|ds+\int_0^t\mathbb{E}|\sigma(s, X_s)|^2ds<\infty.$$
		\item $\phi$ is a continuous process of bounded variation satisfying that $\phi_0=0$, and for any $\varrho\in C([0,T]; \mathbb{R})$,
		\begin{equation}\label{eos995}	
			\int_s^t(\varrho_u-X_u)d\phi_u+\int_s^t\psi(X_u)du\le \int_s^t\psi(\varrho_u)du.
		\end{equation}
		\item For any $t\in[0, T]$, 
		$$
		X_t=X_0+\int_0^tb(s, X_s)ds+\int_0^t\sigma(s, X_s)dB_s-\phi_t, \qquad\mathbb{P}-a.s..$$
	\end{enumerate}
\end{definition}

The remark below covers useful properties followed from Definition \ref{ecd01} and Theorem \ref{thm:Properties_of_subdifferential}. The proofs are referred to \cite{cepa1998problame}.
\begin{remark}\label{ros1}
	If both $(X^1, \phi^1)$ and $(X^2, \phi^2)$ are solutions of the SVI \eqref{eqn:svi}, by equation \eqref{eos995}, we have
	$$
	\int_s^t(X^1_u-X^2_u)(d\phi^1_u-d\phi^2_u)\ge 0.
	$$
	The above result also holds for the MVSVI  \eqref{eqn:mvsvi}.
	Next, if $0 \in \operatorname{Int}(D(\partial\psi))$, there exists $m_0>0$ satisfying $\{a: |a|\le m_0\}\subset \operatorname{Int}(D(\partial\psi))$. Then for any $0\le s<t\le T$, we have that
	\begin{equation}\label{eos096}
		m_0|\phi|_s^t\le \int_s^tX_ud\phi_u+\int_s^t |X_u|du+M(t-s), \qquad M=\sup_{|x|\le m_0}|\psi(x)|,
	\end{equation}
	where $|\phi|_s^t$ stands for the bounded variation of $\phi$ on $[s,t]$.
\end{remark}

\begin{assumption}
	\label{assumption1}
	We impose the following conditions:
	\begin{enumerate}
		\setlength\itemsep{0.18em}
		\item[(1)]	For any $x, x'\in\mathbb{R}$ and $t\in \mathbb{R}^+$, suppose that
		$b(\cdot,x)$ and $\sigma(\cdot,x)$ are measurable and
		there exists a constant $C>0$ such that for some $l>0$ and some $p_0\geq 4l+4$,
		$$|b(t, x)|\le C\big(1+|x|^{l+1})\quad\text{and}\quad
		2xb(t, x)+(p_0-1)|\sigma(t, x)|^2\le C(1+|x|^2).$$
		Furthermore, if $|x| \vee|y| \leq R$ for some $R>0$, there exists a constant $L_R>0$ such that
		$$|b(t, x)-b(t, x')| \leq L_R|x-x'|\quad\text{and}\quad |\sigma(t, x)-\sigma(t, x')| \leq L_R|x-x'|^{\alpha+\frac{1}{2}},$$
		where $\alpha \in\left[0, \frac{1}{2}\right]$.
		\smallskip
		
		\item[(2)]			Suppose $\psi(x)$ is lower semicontinuous satisfying that $0 \in \operatorname{Int}(D(\partial\psi))$ and $\psi(x) \geq \psi(0)=0$ for any $x \in \mathbb{R}$.
		The initial state $X_0 \in \overline{D(\partial(\psi))}$ a.s., and
		$$
		\mathbb{E}\left|X_0\right|^{p_0}<+\infty \quad\text{and}\quad \mathbb{E}\, \psi^2(X_0)<+\infty.
		$$
		
	\end{enumerate}
\end{assumption}
\begin{remark}
	By the above assumption, there exists a constant $C>0$ such that
	\begin{equation}
		\label{eqn:eqn1}
		(p_0-1)|\sigma(t,x)|^2\leq C(1+|x|^2)+2|xb(t,x)|\leq C(1+|x|^{2+l}).
	\end{equation}
	One toy example with coefficients $b$ and $\sigma$ satisfying Assumption \ref{assumption1} is given below
	$$
	dX_t\in (X_t-2X_t^3)dt+|X_t^2+X_t|^{\frac12+\alpha}dB_t-\partial\psi(X_t)dt,
	$$
	where $\alpha\in[0,\frac12)$. 
\end{remark}

\begin{theorem}
	\label{thm:svi_wellposedness}
	Under Assumption \ref{assumption1},  there exists a unique strong solution $(X, \phi)$ to equation \eqref{eqn:svi}.
	Moreover, for any $0<p \leq p_0-l$,
	$$
	\mathbb{E} \sup_{t \leq T}\left|X_t\right|^p \leq C_{p, T}\left(1+\mathbb{E}\left|X_0\right|^p\right)\quad\text{and}\quad \mathbb{E}\left(|\phi|_0^T\right)^{p/2} \leq C_{p, T}(1+\mathbb{E}|X_0|^p).
	$$
\end{theorem}
\begin{proof}
	We complete the proof by proceeding with the following $6$ steps. We will apply the Yosida-Moreau approximation of $\psi$ using $\psi^n(x)$ defined in equation \eqref{eqn:Yosida-Moreau approximation}.
	The gradient of $\psi^n$, denoted as $\nabla\psi^n$, satisfies the properties listed in Theorem \ref{thm:Yosida-Moreau}.
	Then, consider the following SDE which replaces $\partial\psi$ in equation \eqref{eqn:svi}  by $\nabla \psi^n$:
	\begin{equation}
		\label{eqn:eqn2}
		X_t^n=X_0+\int_0^t b\left(s, X_s^n\right) d s+\int_0^t \sigma\left(s, X_s^n\right) d B_s-\int_0^t \nabla \psi^n (X_s^n) d s.
	\end{equation}
	It then follows from Theorem $2.1$ of \cite{ngo2019tamed} that for every $n \geq 1$, a unique strong solution $X^n$ of equation \eqref{eqn:eqn2} exists.
	\smallskip
	
	\noindent\textbf{Step $1$.} In this step, we aim to bound $\mathbb{E}\sup_{t \leq T}\left(1+|X_{t}^{n}|^{2}\right)^{\frac{q}{2}}$ for any $q \in\left[2, p_{0}-l\right]$.
	Applying It\^o's formula, for $p \leq p_{0}$, we have
	\begin{align}
		\label{eqn:delta}
		\left(1+|X_{t}^{n}|^{2}\right)^{\frac{p}{2}} &\leq\left(1+\left|X_{0}\right|^{2}\right)^{\frac{p}{2}} +p \int_{0}^{t}(1+|X_{s}^{n}|^{2})^{\frac{p}{2}-1} X_{s}^{n} \sigma\left(s, X_{s}^{n}\right) d B_{s}\nonumber\\
		& \quad+\frac{p}{2} \int_{0}^{t}(1+|X_{s}^{n}|^{2})^{\frac{p}{2}-1}\left[2 X_{s}^{n} b\left(s, X_{s}^{n}\right)+(p-1)\left|\sigma\left(s, X_{s}^{n}\right)\right|^{2}\right] ds \nonumber\\
		& \quad-p \int_{0}^{t}(1+|X_{s}^{n}|^{2})^{\frac{p}{2}-1} X_{s}^{n} \nabla \psi^{n}\left(X_{s}^{n}\right) d s \nonumber\\
		&\leq  (1+|X_{0}|^{2})^{\frac{p}{2}}+p \int_{0}^{t}(1+|X_{s}^{n}|^{2})^{\frac{p}{2}-1} X_{s}^{n} \sigma\left(s, X_{s}^{n}\right) d B_{s}\nonumber\\
		& \quad+\frac{p}{2} \int_{0}^{t}(1+|X_{s}^{n}|^{2})^{\frac{p}{2}-1}\left[2 X_{s}^{n} b\left(s, X_{s}^{n}\right)+(p-1)\left|\sigma\left(s, X_{s}^{n}\right)\right|^{2}\right] d s,
	\end{align}
	where we used $\psi^n(x)\geq0$ and
	\begin{align}
		\label{eqn:eqn3}
		-x \nabla \psi^{n}(x) \leq \psi^{n}(0)-\psi^{n}(x) \leq \psi^{n}(0)=0.
	\end{align}
	For $R>0$, set $$\tau_{R}^{n}:=\inf\big\{t\geq 0 ; |X_{t}^{n}|>R\big\}.$$ Taking expectations on both sides of equation \eqref{eqn:delta} gives
	\begin{align*}
		\mathbb{E}\left(1+ |X_{t\wedge \tau_{R}^{n}}^{n}|^{2}\right)^{\frac{p}{2}} & \leq \mathbb{E}\left(1+ |X_{0}|^{2}\right)^{\frac{p}{2}}+C \mathbb{E} \int_{0}^{t\wedge \tau_{R}^{n}} \frac{p}{2}(1+|X_{s}^{n}|^{2})^{\frac{p}{2}-1}(1+|X_{s}^{n}|^{2}) d s \\
		& \leq \mathbb{E}(1+|X_{0}|^{2})^{\frac{p}{2}}+\frac{p C}{2} \int_{0}^{t} \mathbb{E}\left(1+|X_{s \wedge \tau_{R}^{n}}^{n}|^{2}\right)^{\frac{p}{2}}d s .
	\end{align*}
	Gr\"onwall's lemma yields
	$$
	\mathbb{E}\left(1+|X_{t\wedge \tau_{R}^{n}}^{n}|^{2}\right)^{\frac{p}{2}} \leq e^{\frac{p C t}{2}} \mathbb{E}(1+|X_{0}|^{2})^{\frac{p}{2}}.
	$$
	Sending $R \rightarrow \infty$ gives that for $0<p \leq p_{0}$,
	\begin{align}
		\label{eqn:eqn4}
		\mathbb{E}\left(1+|X_{t}^{n}|^{2}\right)^{\frac{p}{2}}\leq e^{\frac{p C t}{2}} \mathbb{E}(1+|X_{0}|^{2})^{\frac{p}{2}}.
	\end{align}
	
	Applying the Burkholder-Davis-Gundy (BDG) inequality and equation \eqref{eqn:eqn1},
	\begin{align*}
		&\hspace{-1.5cm}\mathbb{E}\sup _{t \leq T}\left|p \int_{0}^{t}(1+|X_{s}^{n}|^{2})^{\frac{p}{2}-1} X_{s}^{n} \sigma\left(s, X_{s}^{n}\right) d B_{s}\right| \\&\leq C_p \mathbb{E}\left\{\int_{0}^{T}(1+|X_{s}^{n}|^{2})^{p-2}|X_{s}^{n}|^{2} \left| \sigma(s, X_{s}^{n})\right|^{2} d s\right\}^{\frac{1}{2}} \\
		& \leq C_p \mathbb{E}\left\{\int_{0}^{T}(1+|X_{s}^{n}|^{2})^{p-1}\left(1+|X_{s}^{n}|^{2}\right)^{\frac{l}{2}+1} d s\right\}^{\frac{1}{2}} \\
		& \leq C_p \mathbb{E}\left\{\int_{0}^{T}\left(1+|X_{s}^{n}|^{2}\right)^{p / 2}(1+|X_{s}^{n}|^{2})^{\frac{p+l}{2}} d s\right\}^{\frac{1}{2}} \\
		& \leq \frac{1}{2} \mathbb{E}\sup_{t \leq T}\left(1+|X_{t}^{n}|^{2}\right)^{\frac{p}{2}}+C_p \mathbb{E}\int_{0}^{T}\left(1+|X_{s}^{n}|^2\right)^{\frac{p+l}{2}} d s  \\
		& \leq \frac{1}{2} \mathbb{E}\sup_{t \leq T}\left(1+|X_{t}^{n}|^{2}\right)^{\frac{p}{2}}+C_p \mathbb{E}\int_{0}^{T}\left(1+|X_{s}^{n}|^{p+l}\right) d s .
	\end{align*}
	Therefore, for any $p \in\left[2, p_{0}-l\right]$, by equation \eqref{eqn:eqn4},
	\begin{align}
		\label{eqn:doublestar}
		&\mathbb{E}\sup_{t \leq T}\left(1+|X_{t}^{n}|^{2}\right)^{\frac{p}{2}}\nonumber\\
		& \leq \mathbb{E}\left(1+|X_{0}^n|^2\right)^{\frac{p}{2}}+C_p \mathbb{E}\int_{0}^{T}\left(1+|X_{s}^n|^{2}\right)^{\frac{p}{2}} d s+C_p \mathbb{E}\int_{0}^{T}\left(1+|X_{s}^n|^{p+l}\right)d s \nonumber\\
		& \leq C_{p, T}\left(1+\mathbb{E}|X_{0}^n|^{p_0}\right).
	\end{align}
	
	\noindent\textbf{Step $2$.} 
	Since $0 \in \operatorname{Int}(D(\partial \psi))$, as stated in Remark \ref{ros1}, there exists $m_0 > 0$ such that $\{ y : |y| \leq m_0 \} \subset \operatorname{Int}(D(\partial \psi))$. Then for any $x\in \mathbb{R}$, we have
	\begin{equation*}
		(y-x)\nabla\psi^n(x)\le \psi^n(y)-\psi^n(x)\le \psi(y).
	\end{equation*}
	It follows that
	\begin{equation}\label{eos997}
		m_0|\nabla\psi^n(x)|\le x\nabla\psi^n(x)+M, \hspace{2em} \forall x\in \mathbb{R},
	\end{equation}
	where $M=\sup\limits_{|y|\le m_0}\psi(y)$. Moveover, for all $t \in[0,T]$, by Assumption \ref{assumption1},
	\begin{align*}
		&2 m_{0} \int_{0}^{t}\left|\nabla \psi^{n}\left(X_{s}^{n}\right)\right| d s\\
		& \leq 2 \int_{0}^{t} X_{s}^{n} \nabla \psi^{n}\left(X_{s}^{n}\right) d s+2 M t \\
		& =|X_{0}|^{2}+2 \int_{0}^{t} X_{s}^{n} b\left(s, X_{s}^{n}\right) d s+\int_{0}^{t}\left|\sigma(s, X_{s}^{n})\right|^{2} d s +2 \int_{0}^{t} X_{s}^{n} \sigma\left(s, X_{s}^{n}\right) d B_{s}-|X_{t}^{n}|^{2}+2 M t \\
		& \leq|X_{0}|^{2}+C \int_{0}^{t}\left(1 +|X_{s}^{n}|^{2}\right) d s-|X_{t}^{n}|^{2}+2 M t +2 \int_{0}^{t} X_{s}^{n} \sigma\left(s, X_{s}^{n}\right) d B_{s},
	\end{align*}
	and hence applying equation \eqref{eqn:eqn4},
	\begin{align}
		\label{eqn:eqn5}
		&\mathbb{E}\left(\int_{0}^{T}\left|\nabla \psi^{n}\left(X_{s}^{n}\right)\right| d s\right)^{2}\nonumber\\
		& \leq C \mathbb{E}|X_{0}|^{4}+C \mathbb{E}\int_{0}^{T}\left(1+|X_{s}^{n}|^{4}\right)ds+CMT +C\mathbb{E}\sup_{t \leq T}\left|\int_{0}^{t}X_{s}^{n} \sigma\left(s, X_{s}^{n}\right) d B_{s} \right|^2\nonumber\\
		& \leq C\left(1+\mathbb{E}|X_{0}|^{4}\right)+C M T+C \mathbb{E}\int_{0} ^{T}|X_{s}^{n}|^{2}\left(1+|X_{s}^{n}|^{2 +l}\right) d s \nonumber\\
		& \leq C_{T}\left(1+\mathbb{E}|X_{0}|^{4+l}\right)+C_{M,T}.
	\end{align}
	Note $\psi^n$ is of $\mathcal{C}^1$ with its derivative Lipschitz continuous, we can apply the generalized It\^o's formula (Theorem \ref{tb8}) to $\psi^{n}\left(X_{t}^{n}\right)$. Then applying It\^o's formula to $|\psi^{n}\left(X_{t}^{n}\right)|^2$ and by Assumption \ref{assumption1} and Theorem \ref{thm:Yosida-Moreau}, and noting that $\nabla\psi^n(x)=nx-J_nx$, 
	\begin{align*}
		&\left|\psi^{n}\left(X_{t}^{n}\right)\right|^{2}\\
		& =\left|\psi^{n}\left(X_{0}\right)\right|^{2}+2 \int_{0}^{t} \psi^{n}\left(X_{s}^{n}\right) \nabla \psi^{n}\left(X_{s}^{n}\right) b\left(s, X_{s}^{n}\right) d s+\int_{0}^{t}\left|\nabla \psi^{n}\left(X_{s}^{n}\right)\right|^{2}\left|\sigma\left(s, X_{s}^{n}\right)\right|^{2} d s \\
		& \quad +n \int_{0}^{t} \psi^{n}\left(X_{s}^{n}\right)\left|\sigma\left(s, X_{s}^{n}\right)\right|^{2} d s-2 \int_{0}^{t} \psi^{n}\left(X_{s}^{n}\right)\left|\nabla \psi^{n}\left(X_{s}^{n}\right)\right|^{2} d s \\
		& \quad+2 \int_{0}^{t} \psi^{n}\left(X_{s}^{n}\right) \nabla \psi^{n}\left(X_{s}^{n}\right) \sigma\left(s, X_{s}^{n}\right) d B_{s} \\
		& \leq \left|\psi^{n}\left(X_{0}\right)\right|^{2}+2 n \int_{0}^{t} \psi^{n}\left(X_{s}^{n}\right) X_{s}^{n} b\left(s, X_{s}^{n}\right) d s +3 n \int_{0}^{t} \psi^{n}\left(X_{s}^{n}\right)\left|\sigma\left(s, X_{s}^{n}\right)\right|^{2} d s\\
		& \quad -2n \int_{0}^{t} \psi^{n}\left(X_{s}^{n}\right)J_nX_{s}^{n}b\left(s, X_{s}^{n}\right) d s-2 \int_{0}^{t} \psi^{n}\left(X_{s}^{n}\right)\left|\nabla \psi^{n}\left(X_{s}^{n}\right)\right|^{2} d s \\
		&\quad+2 \int_{0}^{t} \psi^{n}\left(X_{s}^{n}\right) \nabla \psi^{n}\left(X_{s}^{n}\right) \sigma\left(s, X_{s}^{n}\right) d B_{s} \\
		&\leq \left|\psi\left(X_{0}\right)\right|^{2}  +C n \int_{0}^{t} \psi^{n}\left(X_{s}^{n}\right)(1+|X_{s}^{n}|^{2}) d s
		+Cn\int_{0}^{t} \psi^{n}\left(X_{s}^{n}\right)|J_nX_{s}^{n}|(1+|X_{s}^{n}|^{1+l}) d s\\
		&\quad-2 \int_{0}^{t} \psi^{n}\left(X_{s}^{n}\right)\left|\nabla \psi^{n}\left(X_{s}^{n}\right)\right|^{2} d s+2 \int_{0}^{t} \psi^{n}\left(X_{s}^{n}\right) \nabla \psi^{n}\left(X_{s}^{n}\right) \sigma\left(s, X_{s}^{n}\right) d B_{s}.
	\end{align*}
	Thus, we have by using Young's inequality, 
	\begin{align*}
		&\mathbb{E}\sup_{t \leq T}\left|\psi^{n}\left(X_{t}^{n}\right)\right|^{2}  +2 \mathbb{E}\int_{0}^{T} \psi^{n}\left(X_{s}^{n}\right)\left|\nabla \psi^{n}\left(X_{s}^{n}\right)\right|^{2} d s\\
		&\leq \mathbb{E}\left|\psi\left(X_{0}\right)\right|^{2}  +C n \mathbb{E}\int_{0}^{T} \psi^{n}\left(X_{s}^{n}\right)(1+|X_{s}^{n}|^{2+l}) d s \\
		& \quad+C \mathbb{E}\left\{\int_{0}^{T}\left|\psi^{n}\left(X_{s}^{n}\right)\right|^{2}\left|\nabla \psi^{n}\left(X_{s}^{n}\right)\right|^{2}\left(1+|X_{s}^{n}|^{2+l}\right) d s\right\}^{1/2} \\
		&\leq \mathbb{E}\left|\psi\left(X_{0}\right)\right|^{2}  +C n \mathbb{E}\int_{0}^{T} \psi^{n}\left(X_{s}^{n}\right)(1+|X_{s}^{n}|^{2+l}) d s+\frac{1}{2} \mathbb{E}\sup _{t \leq T}\left|\psi^{n}\left(X_{t}^{n}\right)\right|^{2} \\
		& \quad+C \mathbb{E}\int_{0}^{T}\left|\nabla \psi^{n}\left(X_{s}^{n}\right)\right|^{2}\left(1+|X_{s}^{n}|^{2+l}\right) ds .
	\end{align*}
	By Theorem \ref{thm:Yosida-Moreau}, we have that $|\nabla\psi^n(X^n_s)|^2\leq 2n\psi^n(X^n_s)$ and $\psi^n(x)\leq |x||\nabla\psi^n(x)|$. Then 
	\begin{align*}
		&\mathbb{E}\sup_{t \leq T}\left|\psi^{n}\left(X_{t}^{n}\right)\right|^{2}  +2 \mathbb{E}\int_{0}^{T} \psi^{n}\left(X_{s}^{n}\right)\left|\nabla \psi^{n}\left(X_{s}^{n}\right)\right|^{2} d s\\
		& \leq \mathbb{E}\left|\psi\left(X_{0}\right)\right|^{2}+Cn \mathbb{E}\int_{0}^{T} \psi^{n}\left(X_{s}^{n}\right)\left(1+|X_{s}^{n}|^{2+l}\right) d s +\frac{1}{2} \mathbb{E}\sup_{t \leq T}\left|\psi^{n}\left(X_t^{n}\right)\right|^{2}\\
		& \leq \mathbb{E}\left|\psi\left(X_{0}\right)\right|^{2}+\frac{1}{2} \mathbb{E}\sup _{t \leq T}\left|\psi^{n}\left(X_t^n\right)\right|^{2}\\
		&\quad +Cn \mathbb{E}\int_{0}^{T}\left|\psi^{n}\left(X_{s}^{n}\right)\right|^{1 / 3} \left|\nabla \psi^{n}\left(X_{s}^{n}\right)\right|^{2 / 3} |X_{s}^{n}|^{2 / 3} \left(1+|X_{s}^{n}|^{2+l}\right) d s\\
		& \leq \mathbb{E}\left|\psi\left(X_{0}\right)\right|^{2}+\frac{1}{2} \mathbb{E}\sup _{t \leq T}| \psi^{n}(X_t^n)|^{2}+\mathbb{E}\int_{0}^{T}\left|\psi^{n}\left(X_{s}\right)\right|\left|\nabla \psi^{n}\left(X_{s}^{n}\right)\right|^{2} d s\\
		& \quad+C n^{\frac{3}{2}} \mathbb{E}\int_{0}^{T}|X_{s}^{n}|\left(1+|X_{s}^{n}|^{2+l}\right)^{\frac{3}{2}} d s,
	\end{align*}
	where H\"older's inequality is used in the last inequality. Therefore, we obtain
	\begin{align*}
		&\hspace{-0.5cm} \mathbb{E}\sup_{t \leq T}\left|\psi^{n}\left(X_{t}^{n}\right)\right|^{2}+2 \mathbb{E}\int_{0}^{T} \psi^{n}\left(X_{s}^{n}\right)\left|\nabla \psi^{n}\left(X_{s}^{n}\right)\right|^{2} d s \\
		& \leq 2 \mathbb{E}\left(\psi\left(X_{0}\right)\right)^{2}+C n^{3 / 2} \int_{0}^{T} \mathbb{E}\left[|X_{s}^{n}|+|X_{s}^{n}|^{4+\frac{3}{2} l}\right] d s \\
		& \leq 2 \mathbb{E}\left(\psi\left(X_{0}\right)\right)^{2}+C n^{3/2}\left(1+ \mathbb{E}|X_{0}|^{4+\frac{3}{2} l}\right),
	\end{align*}
	and by Theorem \ref{thm:Yosida-Moreau} again we have
	\begin{align}
		\label{eqn:eqn7}
		\mathbb{E}\sup _{t \leq T}\left|\nabla \psi^{n}\left(X_{t}^{n}\right)\right|^{4}
		\leq 4 n^{2} \mathbb{E}\sup _{t \leq T}\left[\psi^{n}\left(X_{t}^{n}\right)\right]^{2} &  \leq 8 n^{2} \mathbb{E}\left[\psi^{2}\left(X_{0}\right)\right]+C n^{7 / 2}\left(1+\mathbb{E}|X_{0}|^{4+\frac{3}{2} l}\right) \nonumber\\
		& \leq C n^{7 / 2}\left(1+\mathbb{E}|X_{0}|^{4+\frac{3}{2} l}+\mathbb{E}\psi^{2}\left(X_{0}\right)\right).
	\end{align}
	
	\noindent\textbf{Step $3$.}  Recall, the Yamada-Watanabe function defined in \eqref{eqn:Yamada-Watanabe_func}, here as follows
	$$
	V_{\epsilon,\delta}(x)=\int_0^x\int_0^y\varphi_{\epsilon,\delta}(z)dzdy.
	$$
	Applying It\^o's formula to $V_{\epsilon, \delta}\left(X_{t}^{n}-X_{t}^{m}\right)$,  by Theorem \ref{thm:Properties_of_YW},
	\begin{align*}
		\left|X_{t}^{n}-X_{t}^{m}\right| & \leq V_{\epsilon, \delta}\left(X_{t}^{n}-X_{t}^{m}\right)+\epsilon \\
		& =\int_{0}^{t} V_{\epsilon, \delta}^{\prime}\left(X_{s}^{n}-X_{s}^{m}\right)\left[b\left(s, X_{s}^{n}\right)-b\left(s, X_{s}^{m}\right)\right] d s\\
		& \quad+\frac{1}{2} \int_{0}^{t} V_{\epsilon, \delta}^{\prime \prime}\left(X_{s}^{n}-X_{s}^{m}\right)\left|\sigma\left(s, X_{s}^{n}\right)-\sigma\left(s, X_{s}^{m}\right)\right|^{2} d s \\
		& \quad+\int_{0}^{t} V_{\epsilon, \delta}^{\prime}\left(X_{s}^{n}-X_{s}^{m}\right)\left[\sigma\left(s, X_{s}^{n}\right)-\sigma\left(s, X_{s}^{m}\right)\right] d B_{s} \\
		& \quad-\int_{0}^{t} V_{\epsilon, \delta}^{\prime}\left(X_{s}^{n}-X_{s}^{m}\right)\left(\nabla \psi^{n}\left(X_{s}^{m}\right)-\nabla \psi^{m}\left(X_{s}^{m}\right)\right) d s+\epsilon \\
		& =: \sum_{i=1}^{4} \mathcal{J}_{4, i}(t)+\epsilon.
	\end{align*}
	For $t<\theta_{n, m}:=\inf\Big\{t>0 ;\,|X_{t}^{n}| \vee\left|X_{t}^{m}\right|>R\Big\} $, we have $$\left|\mathcal{J}_{4,1}(t)\right| \leq \int_{0}^{t} L_{R}\left|X_{s}^{n}-X_{s}^{m}\right| d s$$ and
	\begin{align*}		
		\left|\mathcal{J}_{4,2}(t)\right| \leq \frac{1}{2} \int_{0}^{t} \frac{2}{\ln \delta \left|X_{s}^{n}-X_{s}^{m}\right|}  L_{R}^{2}\left|X_{s}^{n} -X_{s}^{m}\right|^{2 \alpha+1}\mathbbm{1}_{[\frac{\epsilon}{\delta},\epsilon]}(|X_{s}^n-X_{s}^m|) d s
		\leq\frac{t L_{R}^{2} \epsilon^{2 \alpha}}{\ln \delta}.
	\end{align*}
	Given that $\mathcal{J}_{4, 3}$ is a local martingale, we have $\mathbb{E}\mathcal{J}_{4, 3}(t \wedge \theta_{n, m})=0$.
	By Theorem \ref{thm:Yosida-Moreau}, 
	we can rewrite $\mathcal{J}_{4,4}$ as
	\begin{align*}
		\mathcal{J}_{4,4}(t)& =-\int_{0} ^{t} V_{\epsilon, \delta}^{\prime}\left(\left|X_{s}^{n}-X_{s}^{m}\right|\right) \frac{X_{s}^{n}-X_{s}^{m}}{\left|X_{s}^{n}-X_{s}^{m}\right|}\left(\nabla \psi^{n}\left(X_{s}^{n}\right)-\nabla \psi^{m}\left(X_{s}^{m}\right)\right) d s\\
		& \leq\int_{0}^{t} \frac{\delta}{\epsilon}\left(\frac{1}{n}+\frac{1}{m}\right) \nabla \psi^{n}\left(X_{s}^{n}\right)  \nabla \psi^{m}\left(X_{s}^{m}\right) d s.
	\end{align*}
	By H\"older's inequality and equations \eqref{eqn:eqn5}-\eqref{eqn:eqn7},
	\begin{align*}
		&\mathbb{E}\left|X_{t \wedge \theta_{n, m}}^{n}-X_{t \wedge \theta_{n, m}}^{m}\right| \\
		& \leq \frac{\delta}{\epsilon}\left(\frac{1}{n}+\frac{1}{m}\right) \mathbb{E}\int_{0} ^{t}\left|\nabla \psi^{n}\left(X_{s}^{n}\right)\right| \left|\nabla \psi^{m}\left(X_{s}^{m}\right)\right| d s+C L_{R} \mathbb{E}\int_{0}^{T}\left|X_{s \wedge \theta_{n, m}}^{n}-X_{s \wedge \theta_{n, m}}^{m}\right| d s\\
		&\quad+C t L_{R}^{2} \epsilon^{2 \alpha}+\epsilon\\
		& \leq \frac{\delta}{\epsilon}\frac{1}{n}\left\{ \mathbb{E}\sup_{s\leq t}|\nabla\psi^n(X_s^n)|^2\mathbb{E}\left(\int_{0}^{t}|\nabla \psi^m(X_s^m)|ds\right)^2  \right\}^{1/2}\\
		&\quad+\frac{\delta}{\epsilon} \frac{1}{m}\left\{\mathbb{E} \sup _{s\leq t}\left|\nabla \psi^{m}\left(X_{s}^{m}\right)\right|^{2}\mathbb{E}\left(\int_{0}^{t}\left|\nabla \psi^{n}\left(X_{s}^{n}\right)\right| d s\right)^{2}\right\}^{1/2}\\
		& \quad+C L_{R} \mathbb{E}\int_{0}^{T}\left|X_{s \wedge \theta_{n, m}}^{n}-X_{s \wedge \theta_{n, m}}^{m}\right| d s+C t L_R^{2} \epsilon^{2 \alpha}+\epsilon\\
		& \leq \frac{C \delta}{\epsilon}\left(n^{-\frac{1}{8}}+m^{-\frac{1}{8}}\right)\left(1+\mathbb{E}|X_{0}|^{4+{\frac{3l}{2}} }+\mathbb{E}\psi^{2}(X_0)\right)^{\frac{1}{4}}+C t L_{R}^{2} \epsilon^{2 \alpha}+\epsilon \\
		&\quad+C L_R \int_{0}^{t} \mathbb{E}\left|X_{s \wedge \theta_{n, m}}^{n}-X_{s \wedge \theta_{n, m}}^{m}\right| d s,
	\end{align*}
	which by Gr\"onwall's lemma yields that
	\begin{align*}
		\mathbb{E}\left|X_{t \wedge \theta_{n, m}^{n}}^{n}-X_{t \wedge \theta_{n, m}}^{m}\right|
		\leq e^{CL_R t}\left[\frac{\delta}{\epsilon}\left(n^{-\frac{1}{8}}+m^{-\frac{1}{8}}\right)\left(1+\mathbb{E}|X_{0}|^{4+\frac{3l}{2}}+\mathbb{E}\psi^{2}\left(X_{0}\right)\right)^{\frac{1}{4}}+\epsilon^{2 \alpha}\right].
	\end{align*}
	Next, we have
	\begin{align*}
		&\mathbb{E}\sup_{t \leq T}|\mathcal{J}_{4,3}(t \wedge \theta_{n, m})|\\
		&\leq \mathbb{E}\left(  \int_{0}^{T}\Big|\sigma(s,X_{s \wedge \theta_{n, m}}^n)-\sigma(s,X_{s \wedge \theta_{n, m}}^m)\Big|^{2}ds\right)^{1/2}\\
		&\leq
		L_R \mathbb{E}\left(\int_{0}^{T}\left|X_{s \wedge \theta_{n, m}}^{n}-X_{s \wedge \theta_{n, m}}^{m}\right|^{1+2 \alpha} d s\right)^{1/2} \\
		&\leq \frac{1}{2} \mathbb{E} \sup_{t \leq  T}\left|X_{t}^{n}-X_{t}^{m}\right|+
		\frac{L_R}{2} \mathbb{E}\int_{0}^{T}\left|X_{s \wedge \theta_{n, m}}^{n}-X_{s \wedge \theta_{n, m}}^{m}\right|^{2 \alpha} d s \\
		& \leq \frac{1}{2} \mathbb{E} \sup_{t \leq  T}\left|X_{t}^{n}-X_{t}^{m}\right|+e^{C L_{R} T}\left[\frac{\delta}{\epsilon}\left(n^{-\frac{1}{8}}+m^{-\frac{1}{8}}\right)\left(1+\mathbb{E}|X_{0}|^{4+\frac{3l}{2} }+\mathbb{E}\psi^2{\left(X_{0}\right)}\right)^{\frac{1}{4}}+\epsilon^{2 \alpha}\right]^{2\alpha}.
	\end{align*}
	Then we have
	\begin{align*}
		&\mathbb{E}\sup_{t \leq T}\left|X^n_{t \wedge \theta_{n, m}}-X_{t\wedge\theta_{n, m}}^{m}\right|\\
		& \leq C \mathbb{E}\left\{\int_{0}^{T\wedge \theta_{n, m}} |\sigma(s,X_s^n) -\sigma(s,X_{s}^{m})|^{2} d s\right\}^{\frac{1}{2}} \\
		& \leq CL_R \mathbb{E}\left\{\int_{0}^{T\wedge \theta_{n, m}}\left|X_{s}^{n}-X_{s}^{m}\right|^{2 \alpha+1} d s\right\}^{\frac{1}{2}}\\
		& \leq \frac{1}{2} \mathbb{E}\sup_{t \leq T}\left|X_{t\wedge\theta_{n, m}}^{m}-X_{t\wedge\theta_{n, m}}^{n}\right|+CL_{R} \mathbb{E} \int_{0}^{T \wedge \theta_{n, m}}\left|X_{s}^{n}-X_{s}^{m}\right|^{2\alpha} d s.
	\end{align*}
	
	Summing up the above estimates and using Young's inequality, we obtain
	\begin{align*}
		\mathbb{E}\sup_{t \leq T}\left|X_{t}^{n}-X_{t}^{m}\right|
		& \leq \mathbb{E}\sup_{t \leq T}\left|X_{t \wedge \theta_{n, m}^{n}}^{n}-X_{t \wedge \theta_{n, m}^{m}}^{m}\right|+\mathbb{E}\sup _{t \leq T}\left|X_{t}^{n}-X_{t}^{m}\right|  \mathbbm{1}_{\left\{T \geq \theta_{n, m}\right\}} \\
		& \leq e^{C L_{R} T}\left[\frac{\delta}{\epsilon}\left(n^{-\frac{1}{8}}+m^{-\frac{1}{8}}\right)\left(1+\mathbb{E}|X_{0}|^{4+\frac{3l}{2}}+\mathbb{E}\psi^{2}\left(X_{0}\right)\right)^{\frac{1}{4}}+\epsilon^{2 \alpha}\right]^{2 \alpha} \\
		& \quad+\frac{1}{2} \mathbb{E}\sup _{t \leq T}\left|X_{t}^{n}-X_{t}^{m}\right| +\frac{1}{2 R} \mathbb{E}\sup_{t\leq T}\left|X_{t}^{n}-X_{t}^{m}\right|^{2}+\frac{R}{2} \mathbb{P}\left(T \geq \theta_{n, m}\right).
	\end{align*}
	Then by equation \eqref{eqn:doublestar}, we have
	\begin{align*}
		\mathbb{E}\sup_{t \leq T}\left|X_{t}^{n}-X_{t}^{m}\right| &\leq e^{CL_R T}\left[\frac{\delta}{\epsilon}\left(n^{-\frac{1}{8}}+m^{-\frac{1}{8}}\right)\left(1+\mathbb{E}|X_{0}|^{4+\frac{3l}{2} }+\mathbb{E}\psi^{2}\left(X_{0}\right)\right)^{\frac{1}{2}}+\epsilon^{2 \alpha}\right]^{2 \alpha}\\
		&\quad +\frac{C}{R}\left(1+\mathbb{E}|X_{0}|^{2}\right)\\
		&\rightarrow 0 \qquad \text{first sending } n,m\rightarrow+\infty \;\; \text{and then } \epsilon\rightarrow0,\;\;R\rightarrow +\infty.
	\end{align*}
	Hence, $\left\{X^{n}\right\}$ is Cauchy in $L^{1}(\Omega ; C([0, T] ; \mathbb{R}))$.
	\medskip
	
	\noindent\textbf{Step $4$.}  Set $\phi_{t}^{n}:=\int_{0}^{t} \nabla \psi^{n}\left(X_{s}^{n}\right) d s$. Then by equation \eqref{eqn:eqn5},
	\begin{align}
		\label{eqn:eqn9}
		\sup _{n} \mathbb{E}\left(\left|\phi^{n}\right|_0^T\right)^{2} \leq C_{T}\left(1+\mathbb{E}|X_{0}|^{4+l}\right),
	\end{align}
where $\left|\phi^{n}\right|_0^T$ denotes the total variation of $\phi^{n}$ on $[0,T]$.
	Note that
	$$\phi_{t}^{n}=X_{0}+\int_{0}^{t} b\left(s, X_{s}^{n}\right) d s+\int_{0}^{t} \sigma\left(s, X_{s}^{n}\right) d B_{s}-X_{t}^{n}.$$
	It follows from Young's inequality, and equations \eqref{eqn:doublestar} and \eqref{eqn:eqn9} that
	\begin{align*}
		\mathbb{E}\sup_{t \leq T}\left|\phi_{t}^{n}-\phi_{t}^{m}\right|
		&\leq \mathbb{E}\sup_{t \leq T}\left|\phi_{t}^{n}-\phi_{t}^{m}\right| \mathbbm{1}_{\{T < \theta_{n, m}\}}+\mathbb{E}\sup_{t \leq T}\left|\phi_{t}^{n}-\phi_{t}^{m}\right| \mathbbm{1}_{\{T \geq \theta_{n, m}\}}\\
		& \leq \mathbb{E} \int_{0}^{T \wedge \theta_{n, m}} | b(s, X_{s}^{n})-b\left(s, X_{s}^{m}\right)|d s\\
		& \quad+\mathbb{E}\sup_{t \leq T}\left|\int_{0} ^{t \wedge \theta_{n, m}}\left[\sigma\left(s, X_{s}^{n}\right)-\sigma(s, X_{s}^{m})\right] d B_{s}\right|\\
		& \quad+\mathbb{E}\sup_{t \leq T}\left|X_{t \wedge \theta_{n, m}}^{n}-X_{t \wedge \theta_{n, m}}^{m}\right|+\mathbb{E}\sup_{t \leq T}\left|\phi_{t}^{n}-\phi_{t}^{m}\right| \mathbbm{1}_{\{T \geq \theta_{n, m}\}}\\
		& \leq L_{R} \mathbb{E} \int_{0}^{T \wedge \theta_{n, m}}\left|X_{s}^{n}-X_{s}^{m}\right| d s+C L_{R} \mathbb{E}\left(\int_{0}^{T\wedge \theta_{n, m}}\left|X_{s}^{n}-X_{s}^{m}\right|^{2 \alpha+1} d s \right)^{1/2}\\
		& \quad+\mathbb{E}\sup_{t \leq T}\left|X_{t \wedge \theta_{n, m}}^{n}-X_{t \wedge \theta_{n, m}}^{m}\right| +\frac{1}{2 R} \mathbb{E}\sup _{t \leq T}|\phi_{t}^{n}-\phi_{t} ^{m}|^{2}+\frac{R}{2} \mathbb{P}\left(T \geq \theta_{n, m}\right).
	\end{align*}
	Then by equation \eqref{eqn:eqn9},
	\begin{align*}
		\mathbb{E}\sup_{t \leq T}\left|\phi_{t}^{n}-\phi_{t}^{m}\right|
		& \leq L_{R} \mathbb{E} \int_{0}^{T \wedge \theta_{n, m}} \left|X_{s}^{n}-X_{s}^{m}\right| d s+C L_{R} \mathbb{E}\int_{0} ^{T \wedge \theta_{m m}}\left|X_{s}^{n}-X_{s}^{m}\right|^{2 \alpha} d s\\
		& \quad+\frac{3}{2} \mathbb{E} \sup \left|X_{t \wedge \theta_{n, m}}^{n}-X_{t \wedge \theta_{n, m}}^{m}\right|+\frac{C_{T}}{ R}\left(1+\mathbb{E}|X_{0}|^{2}\right),
	\end{align*}
	which goes to zero by sending $n, m\to\infty$ and then $R\to\infty$.
	\medskip
	
	\noindent\textbf{Step $5$.} In this step, we will finish the proof of existence. By Steps $3$ and $4$, we know that  $\left(X^{n}, \phi^{n}\right)_n$ is Cauchy in $L^1(\Omega ; C([0,T];\mathbb{R}))$. Hence, there exists a pair of processes $(X, \phi)$ in $L^1(\Omega ; C([0,T];\mathbb{R}))$, such that
	\begin{align}
		\label{eqn:eqn10}
		\lim _{n \rightarrow \infty} \mathbb{E}\sup_{t \leq T}\left|X_{t}^{n}-X_{t}\right|=0 \quad \text{and}\quad\lim _{n \rightarrow \infty} \mathbb{E}\sup_{t \leq T}\left|\phi_{t}^{n}-\phi_{t}\right|=0.
	\end{align}
	It follows from equation \eqref{eqn:eqn5} that
	\begin{align}
		\label{eqn:eqn11}
		\mathbb{E}|\phi|_0^T \leq \sup _{n} \mathbb{E}|\phi^{n}|_0^T<C_{T}\left(1+\mathbb{E}|X_{0}|^{4+l}\right)^{1 / 2}<+\infty.
	\end{align}
	By equations \eqref{eqn:eqn10} and \eqref{eqn:eqn11}, there exists a subsequence which we still denote by $\{n\}$, and $\Omega_{0}$ with $\mathbb{P}\left(\Omega_{0}\right)=1$ such that for all $\omega \in \Omega_{0}$, we have
	$$|\phi(\omega)|_0^T<\infty,\qquad \sup _{n}\left|\phi^{n}(\omega)\right|_0^T<\infty,$$
	$$
	\sup _{t \leq T}\left|X_{t}^{n}(\omega)-X(\omega)\right| \rightarrow 0\quad\text{and}\quad\sup _{t \leq T}\left|\phi_{t}^{n}(\omega)-\phi_{t}(\omega)\right| \rightarrow 0.
	$$
	Therefore, by Theorem \ref{dd16}, for any $\rho \in C([0, T] ; \mathbb{R})$ and any $0 \leq s \leq t \leq T$,
	\begin{align}
		\label{eqn:eqn12}
		\int_{s}^{t}\left(\rho_{r}-X_{r}^{n}(\omega)\right) d \phi_{r}^{n}(\omega) \longrightarrow \int_{s}^{t}\left(\rho_{r}-X_{r}(\omega)\right) d \phi_{r}(\omega).
	\end{align}
	Noting that $X^{n}$ is a solution of equation \eqref{eqn:eqn2}, we have that for all  $\omega\in \Omega_0$,
	\begin{align}
		\label{eqn:eqn13}
		\int_{s}^{t}\left(\rho_{r}-X_{r}^{n}(\omega)\right) d \phi_{r}^{n}(\omega) & \leq \int_{s}^{t}\left[\psi^{n}\left(\rho_{r}\right)-\psi^{n}\left(X_{r}^{n}(\omega)\right)\right] d r\nonumber\\
		& \leq \int_{s}^{t}\left[\psi^{n}\left(\rho_{r}\right)-\psi\left(J_{n} X_{r}^{n}(\omega)\right)\right] d r \nonumber\\
		& \leq \int_{s}^{t}\left[\psi\left(\rho_{r}\right)-\psi\left(J_{n} X_{r}^{n}(\omega)\right)\right] d r,
	\end{align}
	where the second inequality holds because
	$$
	\psi^{n}(x)=\psi\left({J}_{n} x\right)+\frac{1}{2 n}|\nabla \psi^{n} (x)|^{2}.
	$$
	With the result that $\mathbb{E}\sup_{t \leq T}\left|X_{t}^{n}-X_{t}^{m}\right| \rightarrow 0$ obtained in Step $3$, by
	equation \eqref{eqn:eqn10},
	\begin{align*}
		\mathbb{E}\sup_{t\leq T}\left|J_{n} X_{t}^{n}-X_{t}\right| & \leq \mathbb{E}\sup_{t \leq T}\Big[\left|J_{n} X_{t}^{n}-X_{t}^{n}\right|+\left|X_{t}^{n}-X_{t}\right|\Big] \\
		& \leq \frac{1}{n} \mathbb{E}\sup_{t \leq T}\left|\nabla \psi^{n}\left(X_{t}^{n}\right)\right|+ \mathbb{E}\sup_{t\leq T}\left|X_{t}^{n}-X_{t}\right| \longrightarrow 0 \qquad\qquad\text { as } n \rightarrow \infty .
	\end{align*}
	Hence, there exists a sub-subsequence (still denoted by $\{n\}$) such that $\sup_{t\leq T}\left|J_{n} X_{t}^{n}-X_{t}\right| \rightarrow 0$ a.s., and
	\begin{align*}
		\sup _{t\leq T}\left|J_{n} X_{t}-X_{t}\right| & \leq \sup_{t\leq T}\Big(\left|J_{n} X_{t}-J_{n} X_{t}^{n}\right|+\left|J_{n} X_{t}^{n}-X_{t}\right|\Big) \\
		& \leq \sup_{t\leq T}\Big(\left|X_{t}^{n}-X_{t}\right|+\left|J_{n} X_{t}^{n}-X_{t}\right|\Big)  \longrightarrow 0 \qquad\qquad\text { as } n \rightarrow \infty .
	\end{align*}
	It then follows from the lower semi-continuity of $\psi$ that 
	\begin{align}\label{eqn:eqn14}
		\int_{s}^{t} \psi\left(X_{r}(\omega)\right) d r \leq \liminf_{n} \int_{s}^{t} \psi\left(J_{n} X_{r}^{n}(\omega)\right) d r\;\; a.s..
	\end{align}
	Sending $n \rightarrow \infty$ in equation \eqref{eqn:eqn13} and applying equations \eqref{eqn:eqn12} and \eqref{eqn:eqn14}, we obtain that for any $\omega \in \Omega_{0}$,
	$$
	\int_{s}^{t}\left(\rho_{r}-X_{r}(\omega)\right) d \phi_{r}(\omega) \leq \int_{s}^{t}\left[\psi\left(\rho_{r}\right)-\psi\left(X_{r}(\omega)\right)\right] d r .
	$$
	This inequality along with equation \eqref{eqn:eqn11} yields that $(X, \phi)$ is a solution of equation \eqref{eqn:svi}.
	Moreover, by the arguments in Step $1$, similar to equation \eqref{eqn:doublestar}, we have that for any $0<p<p_0-l$,
	\begin{align}\label{supex2}
		\mathbb{E}\sup_{t \leq T}(1+\left|X_{t}\right|^{2})^{\frac{p}{2}} \leq C_{p, T}\left(1+\mathbb{E}|X_{0}|^{p}\right).
	\end{align}
	Moreover, note that 
	\begin{align*}
		\left|X_{t}\right|^{2}&=\left|X_0\right|^{2} +2 \int_{0}^{t} X_{s} b\left(s, X_{s}\right) d s 
		- 2\int_{0}^{t}X_{s} d \phi_s+ \int_{0}^{t} |\sigma(s, X_{s})|^{2} d s
		+2\int_{0}^{t}X_{s} \sigma\left(s, X_{s}\right) d B_{s}\\
		&\leq \left|X_0\right|^{2} +2C \int_{0}^{t} (1+|X_{s}|^2) d s - 2|\phi|_0^t+ ct +2\int_{0}^{t}X_{s} \sigma\left(s, X_{s}\right) d B_{s}.
	\end{align*} 
	It then follows by equation \eqref{supex2} and the BDG inequality that 
	\begin{align*}
		\mathbb{E}(|\phi|_0^T|)^{p/2}
		\leq C_{p,T}\left(1+\mathbb{E} |X_0|^{p}\right) +C_p\mathbb{E}\sup_{t\leq T}\left|\int_{0}^{t}X_{s} \sigma\left(s, X_{s}\right) d B_{s}\right|^{p/2}\leq C_{p,T}\left(1+\mathbb{E} |X_0|^{p}\right).
	\end{align*}

	\noindent\textbf{Step $6$.}   It remains to prove uniqueness. Suppose $(\widehat{X}, \widehat{\phi})$ is also a solution of equation \eqref{eqn:svi}. Then for any $0<p \leq p_{0}-l$,
	\begin{align}
		\mathbb{E} \sup _{t\leq T}|\widehat{X}_{t}|^{p} \leq C_{p, T}\left(1+\mathbb{E}|X_{0}|^{p}\right)\quad\text{and}\quad \mathbb{E}\left(|\widehat{\phi}|_0^T\right)^{p/2} \leq C_{p, T}\left(1+\mathbb{E}|X_{0}|^{p}\right) .
	\end{align}
	Applying It\^o's formula to $V_{\epsilon, \delta}(X_{t}-\widehat{X}_{t})$, we obtain
	\begin{align}
		\label{eqn:eqn17}
		|X_{t}-\widehat{X}_{t}| & \leq V_{\epsilon, \delta}(X_{t}-\widehat{X}_{t})+\epsilon \nonumber\\
		& =\int_{0}^{t} V_{\epsilon, \delta}^{\prime}(X_{s}-\widehat{X}_{s})\left[b\left(s, X_{s}\right)-b(s, \widehat{X}_{s})\right] d s \nonumber\\
		& \quad+\int_{0}^{t} \frac{1}{2\ln \delta \cdot |X_{s}-\widehat{X}_{s}|} \left|\sigma\left(s, X_{s}\right)-\sigma(s, \widehat{X}_{s})\right|^{2} \mathbbm{1}_{\left\{|X_{s}-\widehat{X}_{s}| \in\left[\frac{\epsilon}{\delta}, \epsilon\right]\right\}}ds \nonumber\\
		& \quad+\int_{0}^{t} V_{\epsilon, \delta}^{\prime}(X_{s}-\widehat{X}_{s}) \left[\sigma\left(s, X_{s}\right)-\sigma(s, \widehat{X}_{s})\right] d B_{s}\nonumber\\
		& \quad-\int_{0}^{t} V_{\epsilon, \delta}^{\prime}(|X_{s}-\widehat{X}_{s}|) \frac{X_{s}-\widehat{X}_{s}}{|X_{s}-\widehat{X}_{s}|}(d \phi_{s}-d \widehat{\phi}_{s})+\epsilon \nonumber\\
		& \leq \int_{0}^{t} \Big| b(s, X_{s})-b(s, \widehat{X}_{s})\Big| d s\nonumber\\ & \quad+\int_{0}^{t} \frac{1}{\ln \delta \cdot |X_{s}-\widehat{X}_{s}|}\left|\sigma\left(s, X_{s}\right)-\sigma(s, \widehat{X}_{s})\right|^{2}\mathbbm{1}_{\left\{|X_{s}-\widehat{X}_{s}| \in\left[\frac{\epsilon}{\delta}, \epsilon\right]\right\} }ds\nonumber\\
		& \quad+\int_{0}^{t} V_{\epsilon, \sigma}^{\prime}(X_{s}-\widehat{X}_{s})\Big[\sigma\left(s, X_{s}\right)-\sigma(s, \widehat{X}_{s})\Big] d B_{s}+\epsilon.
	\end{align}
	Set $$\tau_{R}:=\inf \big\{t \geq 0;\;\left|X_{t}\right| \vee|\widehat{X}_{t}|>R\big\}.
	$$
	For any $0<t \leq T$,
	\begin{align}
		\label{eqn:eqn18}
		\mathbb{E}\left|X_{t}-\widehat{X}_{t}\right|
		&\leq \mathbb{E}\left|X_{t \wedge \tau_{R}}-\widehat{X}_{t \wedge \tau_{R}}\right|+\mathbb{E}|X_{t}-\widehat{X}_{t}|  \mathbbm{1}_{\left\{t>\tau_{R}\right\}}+\epsilon \nonumber \\
		& \leq 2 L_{R} \mathbb{E}\int_{0} ^{t}\left|X_{s \wedge \tau_{R}}-\widehat{X}_{s \wedge \tau_{R}}\right| d s+\frac{L_{R}^{2}}{\ln \delta} \epsilon^{2 \alpha} t+\frac{1}{2 R} \mathbb{E}|X_{t}-\widehat{X}_{t}|^{2}+\frac{R}{2} \mathbb{P}\left(t>\tau_{R}\right) \nonumber \\
		& \leq 2 L_{R} \mathbb{E}\int_{0}^{T}\left|X_{s \wedge \tau_{R}}-\widehat{X}_{s  \wedge \tau_{R}}\right| d s+C\left[L_R^{2} \epsilon^{2\alpha}+\frac{1}{R}\left(1+\mathbb{E}|X_{0}|^{2}\right)\right].
	\end{align}
	Note that by equation \eqref{eqn:eqn17} we have
	$$
	\mathbb{E}\left|X_{t \wedge \tau_{R}}-\widehat{X}_{t \wedge \tau_{R}}\right| \leq 2 L_{R} \mathbb{E}\int_{0}^{T}\left|X_{s \wedge \tau_{R}}-\widehat{X}_{s \wedge \tau_{R}}\right| d s+C \epsilon^{2 \alpha}  L_{R}^{2} t + \epsilon.
	$$
	Sending $\epsilon \rightarrow 0$ and applying Gr\"onwall's lemma, we have for every $t \in[0,T]$,
	\begin{align}
		\label{eqn:eqn19}
		\mathbb{E}\left|X_{t \wedge\tau_{R}}-\widehat{X}_{t \wedge\tau_{R}}\right|=0.
	\end{align}
	Plugging equation \eqref{eqn:eqn19} into equation \eqref{eqn:eqn18}, letting $\epsilon\rightarrow0$ and then $R\rightarrow\infty$, we have
	$$
	\mathbb{E}|X_{t}-\widehat{X}_{t}|=0,
	$$
	and 
	\begin{align*}
		\mathbb{E}\sup_{t \leq T}|X_{t}-\widehat{X}_{t}| & \leq C L_{R} \mathbb{E} \int_{0}^{T\wedge \tau_{R}}\left(|X_{s}-\widehat{X}_{s}|+L_{R}|X_{s}-\widehat{X}_{s}|^{2 \alpha}\right) d s\\
		& \quad+CL_{R}^{2} T \epsilon^{2 \alpha}+\mathbb{E}\sup_{t \leq T}|X_{t}-\widehat{X}_{t}| \mathbbm{1}_{\left\{T>\tau_{R}\right\}}\\
		&\rightarrow 0,
	\end{align*}
	by sending $\epsilon\rightarrow0$ and then $R \rightarrow \infty$.
	Hence, we have proved uniqueness.
\end{proof}

\subsection{Well-posedness of the MVSVI}
\label{sec:Well-posedness_MVSVI}
With Theorem \ref{thm:svi_wellposedness}, now we can proceed to establish the existence and uniqueness of solutions to the MVSVI \eqref{eqn:mvsvi} recalled here as follows:
\begin{align*}
	X_t \in X_0+\int_0^tb(s, X_s, \mu_{s})ds+\int_0^t\sigma(s, X_s,\mu_{s})dB_s-\int_0^t\partial\psi(X_s)ds,
\end{align*}
where $b:[0, T]\times\mathbb{R}\times\mathcal{P}(\mathbb{R})\rightarrow\mathbb{R}$ and  $\sigma: [0, T]\times\mathbb{R}\times\mathcal{P}(\mathbb{R})\rightarrow\mathbb{R}$ as measurable stochastic functions, with $\mathcal{P}(\mathbb{R})$ being the space of probability measures on $\mathbb{R}$. 
We first give the definition of its solution.
\begin{definition}
	\label{def:mvsvi}
	A pair of continuous adapted processes $(X, \phi)$ defined on $(\Omega, \mathscr{F}, \{\mathscr{F}_t\}_{t\in [0,T]}, \mathbb{P})$ is called a solution to equation \eqref{eqn:mvsvi} if $(X, \phi)$ satisfies Definition \ref{ecd01} (I) and (III), in addition to 
	the following conditions:
	\begin{enumerate}
		\setlength\itemsep{0.2em}
		
		\item For any $t\in[0, T]$,
		$$\int_0^t\mathbb{E}|b(s, X_s, \mu_{s})|ds+\int_0^t\mathbb{E}|\sigma(s, X_s,\mu_{s})|^2ds<\infty.$$
		\item For any $t\in[0, T]$, 
		$$
		X_t=X_0+\int_0^tb(s, X_s, \mu_{s})ds+\int_0^t\sigma(s, X_s,\mu_{s})dB_s-\phi_t \qquad\mathbb{P}-a.s..$$
		
	\end{enumerate}
\end{definition}


We first define the Wasserstein distance and the Wasserstein space, which will be needed in Assumption \ref{assumption2}.
Define the $p$-th order Wasserstein distance $W_p(\mu,\nu)$ as
$$
W_p(\mu,\nu):=\inf\limits_{\pi\in\mathscr{C}(\mu,\nu)}\left(\int_{\mathbb{R}\times\mathbb{R}}|x-y|^pd\pi(x,y)\right)^\frac{1}{p},
$$
for $p\in [1,\infty)$,
where $\pi$ represents the coupling measure between $\mu$ and $\nu$ and $\mathscr{C}(\mu,\nu)$ is the collection of those measures. The $p$-th order Wasserstein space $\mathcal{P}_p(\mathbb{R})$ is defined as the collection of probability measures with finite $p$-th order Wasserstein distance, i.e.,
$$
\mathcal{P}_p(\mathbb{R}):=\left\{\mu\in\mathcal{P}(\mathbb{R}):\int_{\mathbb{R}}|x|^pd\mu(x)<\infty \right\}.
$$
It is a complete, separable metric space under $W_p(\mu,\nu)$; see, \cite{villani2009optimal}.

\begin{assumption}	\label{assumption2}
	We impose Assumption \ref{assumption1} (2) and  the following conditions:
	There exist constants $C>0$ and $l>0$, such that, for any $x, x'\in \mathbb{R}$, $\mu, \mu'\in \mathcal{P}_{p_0}\left(\mathbb{R}\right)$ and $t \in [0,T]$, for some $p_0\geq1$,
	\begin{align*}
		x b(t, x, \mu) &\leq C\left(1+|x|^{2}\right),\\
		|b( t, x,\mu)| &\leq C\Big(1+|x|^{l+1}+W_{1}\left(\mu,\delta_{0}\right)\Big),\\
		|b( t, x,\mu)-b(t, x', \mu')|
		&\leq C\Big(1+|x|+|x'|+\mu(|\cdot|^{p_0})+\mu'(|\cdot|^{p_0})\Big)\Big[|x-x'|+W_{1}(\mu, \mu')\Big],
	\end{align*}
	where $\delta_{0}$ is the dirac measure at $0$, and
	\begin{align*}
		\sigma^{2}(t, x, \mu) \leq &C,\\
		|\sigma(t, x,\mu)-\sigma(t, x', \mu')|^2
		\leq& C\Big(1+|x|+|x'|+\mu(|\cdot|^{p_0})+\mu'(|\cdot|^{p_0})\Big)\\
		&\hspace{1cm}\times\Big[|x-x'|^{1+2\alpha}+|x-x'|W_{1}(\mu, \mu')\Big],
	\end{align*}
	where $\alpha\in [0,\frac{1}{2}]$. Furthermore, there exists some $a_0>0$ such that $\mathbb{E}\,e^{a_0|X_0|}<\infty$.
\end{assumption}

The following proposition will be handy to the proof of Theorem \ref{thm:mvsvi_wellposedness}. 
\begin{proposition}
	\label{thm:mvsvi_prop}
	Suppose $(X, \phi)$ is a solution of equation \eqref{eqn:mvsvi}. Then under Assumption \ref{assumption2}, for any $t \in[0, T]$, there exists some $0<a_{t}<a_{0}$ such that
	$$
	\sup _{s \in[0, t]} \mathbb{E}e^{a_{t}\left|X_{s}\right|}<+\infty,
	$$
	and for any $p>0$,
	$$
	\mathbb{E}\sup_{t \leq T}\left|X_{t}\right|^{p} \leq C_{p,T}\left(1+\mathbb{E}|X_0|^{p}\right)\quad\text{and}\quad\mathbb{E}\left(|\phi|_0^T\right)^{p/2} \leq C_{p, T}(1+\mathbb{E}|X_0|^p).
	$$
\end{proposition}
\begin{proof}
	Applying It\^o's formula, for any $p\geq 2$,
	\begin{align}\label{pito}
		\left|X_{t}\right|^{p}&=\left|X_0\right|^{p} +p \int_{0}^{t}\left|X_{s}\right|^{p-2} X_{s} b\left(s, X_{s}, \mu_{{s}}\right) d s - p \int_{0}^{t}\left|X_{s}\right|^{p-2} X_{s} d \phi_s\nonumber\\
		& \quad+\frac{1}{2} p(p-1) \int_{0}^{t}\left|X_{s}\right|^{p-2} |\sigma(s, X_{s},\mu_{s})|^{2} d s +p \int_{0}^{t}\left|X_{s}\right|^{p-2} X_{s} \sigma\left(s, X_{s}, \mu_{{s}}\right) d B_{s}\nonumber\\
		&\leq|X_0|^{p}+p \int_{0}^{t}\left|X_{s}\right|^{p-2} C(1+\left|X_{s}\right|^{2}) d s+\frac{C \cdot p(p-1)}{2} \int_{0}^{t}\left|X_{s}\right|^{p-2} d s\nonumber\\
		& \quad+p \int_{0}^{t}\left|X_{s}\right|^{p-2} X_{s} \sigma\left(s, X_{s}, \mu_{s}\right) d B_{s}.
	\end{align}
	Setting $$\tau_{R}:= \inf \big\{t>0 ;\;\left|X_{t}\right|>R\big\},$$
	and taking expectations on both sides of the above equation,
	\begin{align*}
		\mathbb{E}\left|X_{t\wedge \tau_R}\right|^{p}& \leq \mathbb{E}|X_0|^{p}+C\cdot  p \mathbb{E}\int_{0}^{t\wedge \tau_R} (|X_{s}|^{p-2} +|X_{s}|^{p})d s\\
		&\quad+\frac{C\cdot  p(p-1)}{2} \mathbb{E}\int_{0}^{t\wedge \tau_R}|X_{s}|^{p-2}  d s \\
		& \leq \mathbb{E}|X_0|^{p}+C\cdot pt+C \cdot p \mathbb{E}\int_{0}^{t\wedge \tau_R} |X_{s \wedge \tau_{R}}|^{p} d s\\
		&\quad+\frac{C\cdot  p(p-1)}{2} \mathbb{E}\int_{0}^{t\wedge \tau_R}\left[p^{\frac{p}{2}-1}+\frac{|X_{s\wedge \tau_R} |^p}{p}\right] d s \\
		& \leq \mathbb{E}|X_0|^{p}+C\cdot  t p^{\frac{p}{2}+1}+C\cdot  p \int_{0}^{t} \mathbb{E}\left|X_{s \wedge \tau_R}\right|^{p} d s.
	\end{align*}
	Applying Gr\"onwall's lemma and then letting $R\rightarrow \infty$, we have
	$$\mathbb{E}\left|X_{t}\right|^{p} \leq\left(\mathbb{E}|X_0|^{p}+C \cdot t p^{\frac{p}{2}+1}\right) e^{C \cdot p t}, $$
	and furthermore,
	\begin{align*}
		\sup_{s \leq t} \mathbb{E}\,e^{a_{t}\left|X_{s}\right|} & =\sup_{s\leq t} \mathbb{E}\Bigg[\sum_{n=0}^{\infty} \frac{1}{n !} a_{t}^{n}\left|X_{s}\right|^{n}\Bigg]\\
		& =\sup_{s\leq t} \sum_{n=0}^{\infty} \frac{a_{t}^{n}}{n !}\left(e^{C \cdot n s} \mathbb{E}|X_0|^{n}+e^{C \cdot n s}  C \cdot s n^{\frac{n}{2}+1}\right) \\
		& \leq \mathbb{E}\Bigg[\sum_{n=0}^{\infty} \frac{\left(a_{t} e^{C\cdot  t}\right)^{n}}{n !}|X_0|^{n}\Bigg]+\sum_{n=0}^{\infty} \frac{C\cdot  t}{n !} n^{\frac{n}{2}+1} e^{C \cdot n t} a_{t}^{n} \\
		& =\mathbb{E}\Big[e^{a_{t} e^{C t}|X_0|}\Big]+C \cdot t \sum_{n=0}^{\infty} \frac{a_{t}^{n}}{n !} e^{C \cdot n t} n^{\frac{n}{2}+1},
	\end{align*}
	where  if $a_{t} \leq a_{0} e^{-C \cdot t}$ one has
	$$\mathbb{E}\,e^{a_{t} e^{C\cdot  t}|X_0|}<+\infty\quad\text{and}\quad\sum_{n=0}^{\infty} \frac{a_{t}^{n}}{n !} e^{C\cdot  n t} n^{\frac{n}{2}+1}<+\infty.$$
	We have for all $t\in [0,T]$ that $$\sup_{s\leq t} \mathbb{E}e^{a_{t}|X_{s}|}<+\infty.$$ 
	Next, applying the BDG inequality,
	\begin{align*}
		\mathbb{E}\sup _{t\leq T}\left|p \int_{0}^{t}| X_{s}|^{p-2} X_{s} \sigma\left(s, X_{s}, \mu_{s}\right) d B_{s} \right|
		&\leq C_{p} \mathbb{E}\left(\int_{0}^{T}\left|X_{s}\right|^{2 p-2} d s \right)^{1/2}\\
		&\leq \frac{1}{2} \mathbb{E} \sup_{t\leq T}\left|X_{t}\right|^{p}+C_{p} \mathbb{E}\int_{0}^{T}\left|X_{s}\right|^{p-2} d s .
	\end{align*}
	By Gr\"onwall's lemma, we obtain
	\begin{align}\label{supex}\mathbb{E}\sup_{t \leq T}\left|X_{t}\right|^{p} \leq C_{p, T}\left(1+\mathbb{E}|X_0|^{p}\right) .
	\end{align}
	
	At last, applying equation \eqref{eos096} and taking $p=2$ in equation \eqref{pito},
	\begin{align*}
		\left|X_{t}\right|^{2}&=\left|X_0\right|^{2} +2 \int_{0}^{t} X_{s} b\left(s, X_{s}, \mu_{{s}}\right) d s- 2\int_{0}^{t}X_{s} d \phi_s+ \int_{0}^{t} |\sigma(s, X_{s},\mu_{s})|^{2} d s \\
		&\quad
		+2\int_{0}^{t}X_{s} \sigma\left(s, X_{s}, \mu_{{s}}\right) d B_{s}\\
		&\leq \left|X_0\right|^{2} +2C \int_{0}^{t} (1+|X_{s}|^2) d s - 2|\phi|_0^t+ Ct +2\int_{0}^{t}X_{s} \sigma\left(s, X_{s}, \mu_{{s}}\right) d B_{s}.
	\end{align*} 
	It then follows by applying equation \eqref{supex} and the BDG inequality that 
	\begin{align*}
		\mathbb{E}(|\phi|_0^T|)^{p/2}
		&\leq C_{p,T}\left(1+\mathbb{E} |X_0|^{p}\right) +C_p\mathbb{E}\sup_{t\leq T}\Bigg|\int_{0}^{t}X_{s} \sigma\left(s, X_{s}, \mu_{s}\right) d B_{s}\Bigg|^{p/2}\\
		&\leq C_{p,T}\left(1+\mathbb{E} |X_0|^{p}\right).
	\end{align*}
\end{proof}

We apply iteration in distribution to establish the existence. Set $X_t^{(0)}=X_0$ and  define $\mu_{t}^{(n)}:=\mathbb{P}\circ (X_{t}^{(n)})^{-1}$ for $n\geq 0$.
Consider the following equation:
\begin{align}
	\label{eqn:part2_3}
	\left\{\begin{array}{l}
		d X_{t}^{(n+1)} \in b\left(t, X_{t}^{(n+1)}, \mu_{t}^{(n)}\right) d t+\sigma\left(t, X_{t}^{(n+1)}, \mu_{t}^{(n)}\right) d B_{t}-\partial \psi\left(X_{t}^{(n+1)}\right)dt, \\
		X_{0}^{(n+1)}=X_0 .
	\end{array}\right.
\end{align}
Assume $\mu^{(n)}_t$ is well defined, by Theorem \ref{thm:svi_wellposedness}, a unique strong solution $\left(X^{(n+1)}, \phi^{(n+1)}\right)$ of equation \eqref{eqn:part2_3} exists.
Moreover, according to the arguments in the proof of  Proposition \ref{thm:mvsvi_prop}, for every $t \in[0, T]$, there exists $0<a_{t}<a_{0}$ such that
\begin{align}
	\label{eqn:mvsvi_delta}
	\sup_{n} \sup_{s\leq t} \mathbb{E}\, e^{a_{t} |X_{s}^{(n)} |}<+\infty.
\end{align}
For any $p>0$,
\begin{align}
	\label{eqn:mvsvi_delta1}
	\sup_n \mathbb{E}\sup_{t \leq T}|X_{t}^{(n)}|^{p} \leq C_{p,T}  (1+\mathbb{E}|X_0|^{p})\quad\text{and}\quad
	\sup_n \mathbb{E} \left(|\phi^{(n)}|_0^T\right)^{\frac{p}{2}}  \leq C_{p,T}  (1+\mathbb{E}|X_0|^{p}).
\end{align}

Next, for $\overline{\epsilon}>0$ and $R>0$, define
\begin{equation}\begin{split}\label{Ae}
		&I_{\overline{\epsilon},R}:=\left\{ x\in [\tilde{m}_0-R,\tilde{m}_0+R]: d (x, (\overline{D(\partial\psi)})^c) \geq  \overline{\epsilon} \right\},\\
		&h_R(\overline{\epsilon}):=\sup\Big\{|y|:y\in \partial\psi(x),\; x\in I_{\overline{\epsilon},R}\Big\},
\end{split}\end{equation}
where $\tilde{m}_0\in\mathrm{Int}(D(\partial\psi))$ such that $I_{\overline{\epsilon},R}\neq\emptyset$ for every $R>0$ and $\overline{\epsilon}<\overline{\epsilon}_0$ for some $\overline{\epsilon}_0>0$.
Furthermore, for $r>0$, define
\begin{equation}\label{gR}
	g_R(r):=\inf\left\{\overline{\epsilon}\in(0,\overline{\epsilon}_0):h_R(\overline{\epsilon})\leq r^{-1/2}\right\}.
\end{equation}
\begin{proposition}
	\label{thm:mvsvi_prop2}
	For $0 \leq t-s \leq r$, on 
	\begin{align}
		\label{eqn:S_def}
		\mathcal{S}:=\left\{\omega; \;\sup _{t\leq T}|X_t^{(n+1)}(\omega)-\tilde{m}_0| \leq R\right\},
	\end{align}
	we have
	\begin{align*}
		-2\int_{s}^{t}\left(X_{u}^{(n+1)}-X_{s}^{(n+1)}\right) d \phi_{u}^{(n+1)}
		\leq\left(r+g_{R}(r)\right)|\phi^{(n+1)}|_0^T+4Rr h_R\left(r+g_{R}(r)\right).
	\end{align*}
\end{proposition}
\begin{proof}
	Denote  $d(x, A)$ as the Euclidean distance between $x\in \mR$ and $A\subset \mR$.
	Recall, $I_{\overline{\epsilon},R}$ and $h_R(\overline{\epsilon})$ defined in equation \eqref{Ae} of the main text, here as follows:
	\begin{align*}
		&I_{\overline{\epsilon},R}=\left\{ y\in [\tilde{m}_0-R,\tilde{m}_0+R]: d (y, (\overline{D(\partial\psi)})^c) \geq  \overline{\epsilon} \right\}\\
		&h_R(\overline{\epsilon})=\sup\Big\{|y|:y\in \partial\psi(x),\; x\in I_{\overline{\epsilon},R}\Big\}.
	\end{align*}
	Let $\tilde{m}_0\in \mathrm{Int}(D(\partial\psi))$ be chosen
	such that there exists some $\overline{\epsilon}_0>0$ satisfying that for every $R\in \mR^+$ and $\overline{\epsilon}<\overline{\epsilon}_0$,
	$
	I_{\overline{\epsilon},R}\neq\emptyset.$
	Since $\overline{D(\partial\psi)}$ is convex,
	for any $y_1, y_2 \in I_{\overline{\epsilon},R}$
	and  any $\lambda \in (0, 1)$,
	\begin{align*}
		\Big[\lambda y_1 + (1-\lambda)y_2-\overline{\epsilon},\; \lambda y_1 + (1-\lambda)y_2+\overline{\epsilon}\Big]&\subset     \lambda \big(y_1 + [-\overline{\epsilon},\overline{\epsilon}] \big) + (1-\lambda)   \big(y_2 +  [-\overline{\epsilon},\overline{\epsilon}] \big)\\
		&\subset  \overline{D(\partial\psi)}.
	\end{align*}
	Thus $\lambda y_1 + (1-\lambda)y_2 \in I_{\overline{\epsilon},R}$  and $I_{\overline{\epsilon},R}$ is a convex  compact subset of $\mathrm{Int}(D(\partial\psi))$.
	Recall, $g_R$ as the function defined in equation \eqref{gR} of the main text, here as follows:
	$$
	g_R(r)=\inf\left\{\overline{\epsilon}\in(0,\overline{\epsilon}_0):h_R(\overline{\epsilon})\leq r^{-1/2}\right\}.$$
	Then according to the local boundedness of $\partial\psi$ on $\mathrm{Int}(D(\partial\psi))$, 
	$$|h_R(\overline{\epsilon})|<+\infty,\quad
	h_R(r+g_R(r))\leq
	r^{-1/2}\quad\text{and}\quad \lim_{r\downarrow0}g_R(r)=0.$$
	
	Let
	$r_R>0$ be such that $r_R+g_R(r_R)<\overline{\epsilon}_0$. For 
	$r\in(0,r_R\wedge 1]$, we have
	$I_{r+g_R(r),R}\neq\emptyset.$
	Denote by
	$X^{(n+1),r,R}_s$ the projection of $X^{(n+1)}_s$ on
	$I_{r+g_R(r),R}$. On the set $\mathcal{S}$ defined in equation \eqref{eqn:S_def} of the main text recalled here as follows: $$\mathcal{S}=\left\{\omega; \;\sup _{t\leq T}|X_t^{(n+1)}(\omega)-\tilde{m}_0| \leq R\right\},$$ we have
	\begin{align*}
		|X^{(n+1)}_s-X^{(n+1),r,R}_s|\leq r+g_R(r).
	\end{align*}
	Then on the set $\mathcal{S}$, with
	$Y^{(n+1),r,R}_s\in \partial\psi(X_s^{(n+1),r,R})$,
	\begin{align*}
		&2\int_{s}^t( X^{(n+1)}_s-X^{(n+1)}_u)d\phi^{(n+1)}_u\\
		&=2\int_{s}^t( X^{(n+1)}_s-X^{(n+1),r,R}_s)d\phi^{(n+1)}_u+2\int_{s}^t(X^{(n+1),r,R}_s-X^{(n+1)}_u)d\phi^{(n+1)}_u\\
		&\leq 2(r+g_R(r))|\phi^{(n+1)}|_0^T-2\int_{s}^t(X^{(n+1)}_u-X^{(n+1),r,R}_s)(d\phi^{(n+1)}_u-Y^{(n+1),r,R}_sdu) \\
		&\quad-2\int_{s}^t (X^{(n+1)}_u-X^{(n+1),r,R}_s) Y^{(n+1),r,R}_sdu,
	\end{align*}
	where we used Remark \ref{ros1}.
	Then by the boundedness of $ X^{(n+1),r,R}$ and the definitions of $ X^{(n+1),r,R}$ and $h_R(r+g_R(r))$, we have on $\mathcal{S}$,
	\begin{align*}
		2\int_{s}^t( X^{(n+1)}_s-X^{(n+1)}_u)d\phi^{(n+1)}_u
		\leq 2\left(r+g_{R}(r)\right)|\phi^{(n+1)}|_0^T+4 R (t-s) h_{R}\left(r+g_{R}(r)\right),
	\end{align*}
	which completes the proof.
\end{proof}

Now, we are ready to provide the main result of this subsection and its proof.
\begin{theorem}
	\label{thm:mvsvi_wellposedness}
	Under Assumption \ref{assumption2}, there exists a unique strong solution to equation \eqref{eqn:mvsvi}.
\end{theorem}

\begin{proof}
	We first prove the existence in Step (A) and then the uniqueness in Step (B).
	\medskip
	
	\noindent\textbf{Step (A) Existence}
	\smallskip
	
	\noindent\textbf{Step (A.1)}  In this step, we are going to show that $\left(X^{(n+1)}, \phi^{(n+1)}, X^{(n)}, \phi^{(n)}\right)_n$ is
	tight.
	For any $0 \leq s<t \leq T$,
	\begin{align*}
		&\left|X_{t}^{(n+1)}-X_{s}^{(n+1)}\right|^{2}\\
		&=2 \int_{s}^{t}\left(X_{u}^{(n+1)}-X_{s}^{(n+1)}\right) b\left(u, X_{u}^{(n+1)}, \mu_{u}^{(n)}\right) d u-2 \int_{s}^{t}\left(X_{u}^{(n+1)}-X_{s}^{(n+1)}\right) d \phi_{u}^{(n+1)}\\
		& \quad +\int_{s}^{t} \sigma^{2}\left(u, X_{u}^{(n+1)}, \mu_{u}^{(n)}\right)  d u+2 \int_{s}^{t}\left(X_{u}^{(n+1)}-X_{s}^{(n+1)}\right) \sigma\left(u, X_{u}^{(n+1)}, \mu_{u}^{(n)}\right) d B_{u} \\
		& \leq 2 C\int_{s}^{t}\left|X_{u}^{(n+1)}-X_{s}^{(n+1)}\right| \left(1+|X_{u}^{(n+1)}| ^{l+1}+\mathbb{E}|X_{u}^{(n)}|\right) dr\\
		& \quad-2 \int_{s}^{t}\left(X_{u}^{(n+1)}-X_{s}^{(n+1)}\right) d \phi_{u}^{(n+1)}+C(t-s)\\
		& \quad+2 \int_{s}^{t}\left(X_{u}^{(n+1)}-X_{s}^{(n+1)}\right) \sigma\left(u, X_{u}^{(n+1)}, \mu_{u}^{(n)}\right) d B_{u}.
	\end{align*}
	We have, by equation \eqref{eqn:mvsvi_delta} and  Proposition \ref{thm:mvsvi_prop2}, with $\mathcal{S}$ defined in \eqref{eqn:S_def},
	\begin{align*}
		\left|X_{t}^{(n+1)}-X_{s}^{(n+1)}\right|^2\mathbbm{1}_{\mathcal{S}} 
		& \leq 4 C\left(R+\left|\tilde{m}_0\right|\right) \left(1+R^{l+1}+\left|\tilde{m}_0\right|^{l+1}\right)(t-s)\\
		& \quad+2\left[(t-s)+g_{R}(t-s)\right]|\phi^{(n+1)}|_0^T+C(t-s)\\
		& \quad+4 R (t-s)h_R\left(t-s+g_{R}(t-s)\right)\\
		& \quad +2\left|\int_{s}^{t}\left(X_{u}^{(n+1)}-X_{s}^{(n+1)}\right) \sigma\left(u, X_{u}^{(n+1)}, \mu_{u}^{(n)}\right) d B_{u} \right| \mathbbm{1}_{\mathcal{S}} \\
		& \leq C\left(1+R+\left|\tilde{m}_0\right|\right)^{l+2} (t-s)\\
		& \quad+2\left[(t-s)+g_{R}(t-s)\right]|\phi^{(n+1)}|_0^T+C(t-s) \\
		& \quad+4 R (t-s)h_R\left(t-s+g_{R}(t-s)\right)\\
		& \quad+2\left|\int_{s}^{t}\left(X_{u}^{(n+1)}-X_{s}^{(n+1)}\right) \sigma\left(u, X_{u}^{(n+1)}, \mu_{u}^{(n)}\right) d B_{u} \right|  \mathbbm{1}_{\mathcal{S}} ,
	\end{align*}
	and thus for any $\epsilon>0$,
	\begin{align*}
		&\mathbb{P}\left(\sup _{0 \leq t-s \leq r}\left|X_{t}^{(n+1)}-X_{s}^{(n+1)}\right|>\epsilon\right)\\
		& \leq \mathbb{P}\left(\sup_{t\leq T}\left|X_t^{(n +1)}-\tilde{m}_0\right|>R \right)\\
		& \quad
		+\mathbb{P}\left(\sup_{0 \leq t-s \leq r}\left|X_{t}^{(n+1)}-X_{s}^{(n+1)}\right|>\epsilon,\; \sup _{t\leq T}\left|X_{t}^{(n+1)}-\tilde{m}_0\right| \leq R\right) \\
		& \leq \frac{1}{R^{p}} \mathbb{E}\sup_{t \leq T}\left|X_{t}^{(n+1)}-\tilde{m}_0\right|^{p}+\mathbb{P}\left(\big(1+R+\left|\tilde{m}_0\right|\big)^{2+l} r>\frac{\epsilon^2}{3C}\right)\\
		& \quad+\mathbb{P}\left(\big(r+g_{R}(r)\big)|\phi^{(n+1)}|_0^T+2 R r^{\frac{1}{2}}>\frac{\epsilon^2}{6C}\right) \\
		& \quad+\mathbb{P}\left(\sup _{0 \leq t-s \leq r}\left|\int_{s}^{t}\left(X_{u}^{(n +1)}-X_{s}^{(n+1)}\right) \sigma\left(u, X_{u}^{(n+1)}, \mu_{u}^{(n)}\right) d B_{u}\right|>\frac{\epsilon^2}{6},\;\right.
		\\
		& \hspace{8.5cm}\left.\sup _{t\leq T}\left|X_{t}^{(n+1)}-\tilde{m}_0\right| \leq R\right).
	\end{align*}
	Note that by the boundedness od $\sigma$ and \eqref{eqn:mvsvi_delta1}, 
	\begin{align*}
		\sup _{n} \mathbb{E} \sup _{0 \leq |t-s|\leq r}  \int_{s}^{t}\left|X_{u}^{(n+1)}\right|^{2}\left|\sigma\left(u, X_{u}^{(n+1)}, \mu_{u}^{(n)}\right)\right|^{2} d u  \leq C  r \mathbb{E}\sup_{t \leq T} |X_{t}^{(n+1)}|^{2}
	\end{align*}
	and $$\sup_{n} \mathbb{E}\sup _{0 \leq |t-s|\leq r} \int_{s}^{t}\left|\sigma\left(u, X_{u}^{(n+1)}, \mu_{u}^{(n)}\right)\right|^{2} d u \leq C r.$$
	It then follows that $\left\{\int_{0}^{\cdot} X_{u}^{(n+1)} \sigma(u, X_{u}^{(n+1)}, \mu_{u}^{(n)})d B_ r\right\}_{n}$ and $\left\{\int_{0}^{\cdot} \sigma(u, X_{u}^{(n+1)}, \mu_{u}^{(n)}) d B_{u}\right\}_{n}$ are tight and
	$$\limsup_{r \downarrow 0} \sup_{n} \mathbb{P}\left(\sup_{0 \leq t-s \leq r}  \left|\int_{s}^{t}\left(X_{u}^{(n+1)}-X_s^{(n+1)}\right) \sigma\left(u, X_{u}^{(n+1)}, \mu_{u}^{(n)}\right) d B_{u} \right|>\frac{\epsilon^2}{6}\right)=0,$$
	and furthermore,
	\begin{align*}
		\limsup_{r \downarrow 0} \sup_{n} \mathbb{P}\left(\sup_{0 \leq t-s \leq r}\left|X_{t}^{(n+1)}-X_{s}^{(n+1)}\right|>\epsilon\right)=0,\\
		\limsup_{r \downarrow 0} \sup_{n} \mathbb{P}\left(\sup_{0 \leq t-s \leq r} \left|\phi_{t}^{(n+1)}-\phi_{s}^{(n+1)}\right|>\epsilon\right)=0.
	\end{align*}
	Hence, $\left(X^{(n+1)}, \phi^{(n+1)}, X^{(n)}, \phi^{(n)}\right)_n$ is tight, and there exists some $(X, \phi, \widehat{X}, \widehat{\phi})_n$ taking values in $C({[0, T]}; \mathbb{R}^{4})$ and a subsequence of $\left( X^{(n+1)}, \phi^{(n+1)}, X^{(n)}, \phi^{(n)}\right)_{n}$ which converges in distribution to $(X, \phi, \widehat{X}, \widehat{\phi})$.
	\medskip
	
	\noindent\textbf{Step (A.2).}  In this step, we will show that along a subsequence
	$$\mathbb{E}\sup_{t \leq T}\left| X_{t}^{(n+1)}-X_{t}^{(n)}\right| \rightarrow 0,$$ as $n$ goes to infinity.
	Applying the Yamada-Watanabe function in equation \eqref{eqn:Yamada-Watanabe_func} and It\^o's formula,
	\begin{align*}
		&\left|X_{t}^{(n+1)}-X_{t}^{(n)}\right|\\
		& \leq V_{\epsilon, \delta}\left(X_{t}^{(n+1)}-X_{t}^{(n)}\right)+\epsilon \\
		& =\int_{0}^{t} V_{\epsilon, \delta}^{\prime}\left(X_{s}^{(n+1)}-X_{s}^{(n)}\right)\left[b\left(s, X_{s}^{(n+1)}, \mu_{s}^{(n)}\right)-b\left(s, X_{s}^{(n)}, \mu_{s}^{(n-1)}\right)\right] d s \\
		& \quad+\int_{0}^{t} V_{\epsilon, \delta}^{\prime}\left(X_{s}^{(n+1)}-X_{s}^{(n)}\right)\left[\sigma\left(s, X_{s}^{(n+1)}, \mu_{s}^{(n)}\right)-\sigma\left(s, X_{s}^{(n)}, \mu_{s}^{(n-1)}\right)\right] d B_{s}\\
		& \quad-\int_{0}^{t} V_{\epsilon, \delta}^{\prime}\left(X_{s}^{(n+1)}-X_{s}^{(n)}\right) d\left(\phi_{s}^{(n+1)}-\phi_{s}^{(n)}\right) \\
		& \quad+\frac{1}{2} \int_{0}^{t} V_{\epsilon, \delta}^{\prime \prime}\left(X_{s}^{(n+1)}-X_{s}^{(n)}\right)\left|\sigma\left(s, X_{s}^{(n+1)}, \mu_{s}^{(n)}\right)-\sigma\left(s, X_{s}^{(n)}, \mu_{s}^{(n-1)}\right)\right|^{2} d s+\epsilon\\
		&=:\sum_{i=1}^{4} \mathcal{J}_{i}(t)+\epsilon.
	\end{align*}
	Clearly, $\mathbb{E}\mathcal{J}_{2}(t)=0$. Note that by Remark \ref{ros1},
	$$\mathcal{J}_{3}(t)=-\int_{0}^{t} V_{\epsilon, \delta}^{\prime}\left(|X_{s}^{(n+1)}-X_{s}^{(n)}|\right) \frac{X_{s}^{(n+1)}-X_{s}^{(n)}}{|X_{s}^{(n+1)}-X_{s}^{(n)}|} d(\phi_{s}^{(m)}-\phi_{s}^{(n)}) \leq 0.$$
	By Assumption \ref{assumption2} and equation \eqref{eqn:mvsvi_delta1},
	\begin{align*}
		\left|\mathcal{J}_{1}(t)\right| & \leq C \int_{0}^{t}\left(1+|X_{s}^{(n+1)}|+|X_{s}^{(n)}|+\mathbb{E}|X_{s}^{(n)}|^{p_{0}}+\mathbb{E}|X_{s}^{(n-1)}|^{p_{0}}\right)\\
		&\hspace{3.5cm}\times \left(|X_{s}^{(n+1)}-X_{s}^{(n)}|+\mathbb{E}|X_{s}^{(n)}-X_{s}^{(n-1)}|\right)ds \\
		&\leq C \int_{0}^{t} \left(1+\mathbb{E}|X_0|^{p_{0}}+|X_{s}^{(n+1)}|+|X_{s}^{(n)}|\right)\left|X_{s}^{(n+1)}-X_{s}^{(n)} \right|  d s\\
		& \quad+C\int_{0}^{t} \mathbb{E}|X_{s}^{(n)}-X_{s}^{(n-1)}| \left(1+\mathbb{E}|X_0|^{p_{0}}+|X_{s}^{(n+1)}|+|X_{s}^{(n)}|\right) d s\\
		& \leq C_T\int_{0}^{t}(1+2 R_s^{(n+1)})|X_s^{(n+1)}-X_s^{(n)}| ds\\
		& \quad+C\int_{0}^{t}\left(1+|X_s^{(n+1)}|+|X_s^{(n)}|\right)|X_s^{(n+1)}-X_s^{(n)}|\\
		&\hspace{4cm}\times\left(\mathbbm{1}_{\left\{ |X_{s}^{(n+1)}|>R_{s}^{(n+1)}\right\}}+\mathbbm{1}_{\left\{ |X_{s}^{(n)}|>R_{s}^{(n+1)}\right\}}\right) d s\\
		&\quad+C\int_{0}^{t} \mathbb{E}|X_{s}^{(n)}-X_{s}^{(n-1)}| \left(1+|X_{s}^{(n+1)}|+|X_{s}^{(n)}|\right) d s,
	\end{align*} 
	where $R_{s}^{(n+1)}$ was introduced in the last inequality to apply the truncation method.
	Moreover, by equation \eqref{eqn:mvsvi_delta1},
	\begin{align*}
		\left|\mathcal{J}_{4}(t)\right| & \leq\frac{1}{2\ln \delta} \int_{0}^{t}|X_{s}^{(n+1)}-X_{s}^{(n)}|^{-1}\mathbbm{1}_{\{ |X_{s}^{(n+1)}-X_{s}^{(n)}|\in [\epsilon/\delta,\epsilon]  \}}|X_{s}^{(n+1)}-X_{s}^{(n)}| \\
		& \hspace{2cm}\times\Bigg(1+|X_{s}^{(n+1)}|+|X_{s}^{(n)}|+\mathbb{E}|X_{s}^{(n +1)}|^{p_{0}}+\mathbb{E}|X_{s}^{(n)}|^{p_{0}}\Bigg)\\
		&\hspace{2cm}\times \Big[|X_{s}^{(n+1)}-X_{s}^{(n)}|^{2\alpha}+\mathbb{E}|X_{s}^{(n)}-X_{s}^{(n-1)} |\Big]ds\\
		& \leq \frac{C_{T,p_0}}{2\ln \delta} \int_{0}^{t}\left(1+2 R_{s}^{(n+1)}\right)\left[\epsilon^{2 \alpha}+\mathbb{E}|X_{s}^{(n)}-X_{s}^{(n-1)}|\right] d s \\
		&\quad+\frac{C_{T,p_0}}{2\ln \delta} \int_{0}^{t}\left(1+|X_{s}^{(n+1)}|+|X_{s}^{(n)}|\right) \left[|X_{s}^{(n+1)}-X_{s}^{(n)}|^{2 \alpha}+1\right]\\
		&\hspace{4cm}\times \left(\mathbbm{1}_{\left\{ |X_{s}^{(n+1)}|>R_{s}^{(n+1)}\right\}}+\mathbbm{1}_{\left\{ |X_{s}^{(n)}|>R_{s}^{(n+1)}\right\}}\right) d s.
	\end{align*}
	Taking $\delta=2$, we have by \eqref{eqn:mvsvi_delta1} that 
	\begin{align*}
		\mathbb{E}\left|\mathcal{J}_{4}(t)\right| & \leq C_{T, p_{0}} \epsilon^{2\alpha} \int_{0}^{t}(1+2R_{s}^{(n+1)})ds+C_{T, p_{0}} \int_{0}^{t} e^{-\frac{a_T R_{s}^{(n+1)}}{2}} d s \\
		& \quad+C_{T, p_{0}} \int_{0}^{t}(1+2R_{s}^{(n+1)})\mathbb{E}|X_{s}^{(n)}-X_{s}^{(n-1)}|ds,
	\end{align*}
	where the last inequality holds since by equation \eqref{eqn:mvsvi_delta}, for any $t \in[0, T]$,
	$$\mathbb{P}\left(|X_{t}^{(n)}|>R_{t}^{(n)}\right)\leq C_T e^{-a_TR_t^{(n)}.}$$
	Therefore, letting $A_{t}^{(n+1)}:=\sup_{s\leq t}\mathbb{E}|X_{s}^{(n+1)}-X_{s}^{(n)}|,$ 
	for all $0 \leq k \leq n$,
	\begin{align*}
		A_{t}^{(k+1)} & \leq \sup_{s\leq t}\mathbb{E}\left[V_{\epsilon,\delta}(X_{s}^{(k+1)}-X_{s}^{(k)})\right]+\epsilon\\
		&\leq C_{T, p_{0}} \int_{0}^{t}\left(1+2 R_{s}^{(n+1)}\right)\left(A_{s}^{(k+1)}+A_{s}^{(k)}\right) d s+C_{T, p_{0}} \epsilon^{2\alpha}  \int_{0}^{t}\left(1+2 R_{s}^{(n+1)}\right) d s \\
		& \quad+C_{T, p_{0}} \int_{0}^{t} e^{-\frac{a_T R_{s}^{(n+1)}}{2}} d s+\epsilon.
	\end{align*}
	Sending $\epsilon \rightarrow 0$, we obtain
	\begin{align*}
		\sum_{k=0}^{n} A_{t}^{(k)} \leq & C_{T, p_0} \int_{0}^{t}\left(1+\frac{a_T}{2}R_{s}^{(n+1)}\right) \sum_{k=0}^{n} A_{s}^{(k)} d s+C_{T, p_0} n\int_{0} ^{t} e^{-\frac{a_T}{2}R_{s}^{(n+1)}} d s+C_{T, p_{0}}.
	\end{align*}
	Let $$M_n(t):=\frac{\sum_{k=0}^{n} A_{t}^{(k)}}{(n+1)e^2\left( \sup_n\sup_{t \leq T} A_{t}^{(n)}\vee C_{T,p_0}\vee 1 \right)},$$
	and then we have
	\begin{align*}
		M_{n}(t) & \leq \frac{1}{n+1}+C_{T, p_{0}} \int_{0}^{t}\left(1+\frac{a_T}{2} R_{s}^{(n+1)}\right) M_{n}(s) d s+\int_{0}^{t} e^{-\frac{a_T}{2} R_{s}^{(n+1)}} d s.
	\end{align*}
	Setting  $R_{s}^{(n+1)}  =-\frac{2}{a_T}\ln M_n(s)$ yields
	\begin{align*}
		M_{n}(t)&\leq \frac{1}{n+1}+C_{T, p_{0}} \int_{0}^{t}\left(2-\ln M_{n}(s)\right) M_{n}(s) d s \\
		&\leq \frac{1}{n+1}-2 C_{T, p_{0}} \int_{0}^{t} M_{n}(s) \ln M_{n}(s) d s.
	\end{align*}
	Define, for $0<x<e^{-2}$, 
	$$G(x):=\int_{x}^{e^{-2}} \frac{-1}{s \ln s} d s=\ln (-\ln x)-\ln 2.$$
	By Osgood's lemma  (see Theorem \ref{thm:Osgood}), we have
	\begin{align*}
		G\left(M_{n}(t)\right)+C_{T, p_{0}} t \geq G\left(\frac{1}{n+1}\right),
	\end{align*}
	which yields that
	$$M_{n}(t) \leq(n+1)^{-
		e^{tC_{T, p_{0}}}} \quad
	\text{and}
	\quad \frac1n\sum_{k=0}^{n} A_{T}^{(k)} \leq C_{T} n^{-e^{-TC_{T, p_0}}}.
	$$
	By
	Lemma $2.8$ in \cite{erny2022well}, there exists a subsequence of $\{A_{T}^{(n)}\}_{n}$ (which is still denoted by $\{A_T^{(n)}\}_{n}$  with a slight abuse of notation) that converges to $0$, i.e., as $n \rightarrow \infty$,
	\begin{equation}
		\label{eqn:mvsvi_delta3}
		\sup _{t \leq T} \mathbb{E}\left|X_{t}^{(n+1)}-X_{t}^{(n)}\right| \rightarrow 0.
	\end{equation}
	Note that
	\begin{align*}
		&	\mathbb{E}\sup_{t \leq T}\left|\mathcal{J}_{2}(t)\right|\\ & \leq C\mathbb{E}\left\{\int_{0}^{T}\left|\sigma\left(s, X_{s}^{(n+1)}, \mu_{s}^{(n)}\right)-\sigma\left(s, X_{s}^{(n)}, \mu_{s}^{(n-1)}\right)\right|^{2} d s\right\}^{1 / 2} \\
		& \leq C\mathbb{E}\Bigg\{\int_{0}^{T}\left(1+|X_{s}^{(n+1)}|+|X_{s}^{(n)}|\right)|X_{s}^{(n+1)}-X_{s}^{(n)}|\\
		&\hspace{3cm}\times \left(|X_{s}^{(n+1)}-X_{s}^{(n)}|^{2 \alpha}+\mathbb{E}|X_{s}^{(n)}-X_{s}^{(n-1)} |\right)ds\Bigg\}^{1/2} \\
		& \leq \frac{1}{2} \mathbb{E} \sup _{t\leq T}|X_{s}^{(n+1)}-X_{s}^{(n)}|+C\mathbb{E}\int_{0}^{T}\left(1+|X_{s}^{(n+1)}|+|X_{s}^{(n)}|\right) |X_{s}^{(n+1)}-X_{s}^{(n)}|^{2\alpha} d s \\
		& \quad+C \int_{0}^{T} \mathbb{E}|X_{s}^{(n)}-X_{s}^{(n-1)}|\left(1+\mathbb{E}|X_{s}^{(n+1)}|+\mathbb{E}|X_{s}^{(n)}|\right) d s .
	\end{align*}
	Hence,
	\begin{align*}
		&\mathbb{E}\sup_{t \leq T}\left|X_{t}^{(n+1)}-X_{t}^{(n)}\right|\\ &\leq 2 \mathbb{E}\sup_{t \leq T}\left|\mathcal{J}_{1}(t)\right|+2\mathbb{E}\sup_{t \leq T}\left|\mathcal{J}_2(t)\right|+2\mathbb{E}\sup_{t \leq T}\left|\mathcal{J}_4(t)\right|+2 \epsilon \\
		& \leq C_{T}\mathbb{E}\int_{0}^{T}\left(1+2 R_{s}^{(n+1)}\right)|X_{s}^{(n+1)}-X_{s}^{(n)}| d s \\
		& \quad+C_{T} \int_{0}^{T} \mathbb{E} | X_{s}^{(n)}-X_{s}^{(n-1)} |\left(1+\mathbb{E}|X_{s}^{(n)}|+\mathbb{E}|X_{s}^{(n+1)}|\right) d s \\
		& \quad+C_{T} \mathbb{E}\int_{0}^{T}\left(1+|X_{s}^{(n+1)}|+|X_{s}^{(n)}|\right)|X_{s}^{(n+1)}-X_{s}^{(n)}|^{2 \alpha} d s +C_{T, p_{0}} \epsilon^{2 \alpha}\\
		& \quad+C_{T, p_{0}} \int_{0}^{T}\left(1+2 R_{s}^{(n+1)}\right) \mathbb{E}|X_{s}^{(n)}-X_{s}^{(n-1)}| d s+C_{T, p_{0}} \int_{0}^{T} e^{-\frac{a_TR_s^{(n+1)}}{2}} d s+2 \epsilon.
	\end{align*}
	Taking $R_{s}^{(n+1)} \equiv R$ for $R>0$ and then sending $n \rightarrow \infty$, by equations \eqref{eqn:mvsvi_delta1} and \eqref{eqn:mvsvi_delta3}, we have
	\begin{align*}
		\lim _{n \rightarrow \infty} \mathbb{E} \sup_{t\leq T}\left|X_{t}^{(n+1)}-X_{t}^{(n)}\right| \leq C_{T,p_0}\left(\epsilon^{2 \alpha}+\int_{0}^{T} e^{-\frac{a_{T} R}{2}} d s\right) \longrightarrow 0, \qquad \text{as } \epsilon \rightarrow 0,\, R\rightarrow\infty.
	\end{align*}
	Noting that 
	\begin{align*}
		\phi_{t}^{(n+1)}-\phi_{t}^{(n)}&=\int_0^t\left[b(s,X^{(n+1)}_s,\mu^{(n)}_s)-b(s,X^{(n)}_s,\mu^{(n-1)}_s)\right]ds\\
		&\quad+\int_0^t\left[\sigma(s,X^{(n+1)}_s,\mu^{(n)}_s)-\sigma(s,X^{(n)}_s,\mu^{(n-1)}_s)\right]dB_s-\left[X^{(n+1)}_t-X^{(n)}_t\right],
	\end{align*}
	it then follows 
	\begin{align*}
		\lim _{n \rightarrow \infty} \mathbb{E} \sup_{t\leq T}\left|\phi_{t}^{(n+1)}-\phi_{t}^{(n)}\right| =0.
	\end{align*}
	
	\noindent\textbf{Step (A.3).}  From Step (A.2), we know that there exists another subsequence of $\{n \}$, still denoted by $\{n \}$ with a slight abuse of notation, such that
	\begin{align*}
		\sup_{t \leq T}\left|X_t^{(n+1)}-X_t^{(n)}\right| \rightarrow 0\quad\text{and}\quad
		\sup_{t \leq T}\left|\phi_t^{(n+1)}-\phi_t^{(n)}\right| \rightarrow 0 \quad\text { a.s.. }
	\end{align*}
	By Skorohod's representation theorem, there exists a probability space $(\overline{\Omega}, \overline{\mathcal{F}}, \overline{\mathbb{P}})$ on which random variables $\left(Y^{(n+1)}, Y^{(n)}, K^{(n+1)}, K^{(n)}\right)_n$ and $(Y, \widehat{Y}, K, \widehat{K})$ are defined, and
	\begin{align}\label{YKXP}
		& \left(Y^{(n+1)}, Y^{(n)}, K^{(n+1)}, K^{(n)}\right) \stackrel{\mathrm{law}}{=\joinrel=\joinrel=}\left(X^{(n+1)}, X^{(n)}, \phi^{(n+1)}, \phi^{(n)}\right),&\nonumber\\
		& (Y, \widehat{Y}, K, \widehat{K}) \stackrel{\mathrm{law}}{=\joinrel=\joinrel=}(X, \widehat{X}, \phi, \widehat{\phi}), &\nonumber\\
		& \left(Y^{(n+1)}, Y^{(n)}, K^{(n+1)}, K^{(n)}\right) \longrightarrow(Y, \widehat{Y}, K, \widehat{K}) & \overline{\mathbb{P}}-a. s., \nonumber\\
		& \sup _{t\leq T}\left|Y_t^{(n+1)}-Y_t^{(n)}\right| \rightarrow 0,\qquad  \sup _{t\leq T}\left|K_t^{(n+1)}-K_t^{(n)}\right| \rightarrow 0 & \overline{\mathbb{P}}-a.s.,
	\end{align}
	where $ \stackrel{\mathrm{law}}{=\joinrel=\joinrel=}$ means equivalence in distribution.
	Hence, $\overline{\mathbb{P}}\big(Y=\widehat{Y},\; K=\widehat{K}\big)=1.$ Moreover,
	$
	X=\widehat{X}$ and $\phi=\widehat{\phi}$,
	a.s..
	By equation \eqref{eqn:mvsvi_delta1}, $|\phi|_0^T<+\infty$,  a.s., and $|K|_0^T<+\infty$, $\overline{\mathbb{P}}$-a.s..
	Therefore, for any $\rho \in C([0, T] ; \mathbb{R})$ and any $0 \leq s \leq t \leq T$,
	\begin{align*}
		& \int_s^t\big(\rho_{u}-Y_{u}^{(n)}\big) d K_{u}^{(n)} \longrightarrow \int_s^t\left(\rho_{u}-Y_{u}\right) d K_{u} \qquad \overline{\mathbb{P}}-a.s..
	\end{align*}
	Noting that 
	$$
	\int_s^t(\rho_{u}-Y_{u}^{(n)}) d K_{u}^{(n)} \leq \int_s^t[\psi\left(\rho_{u}\right)-\psi(Y_{u}^{(n)})] d u \qquad \overline{\mathbb{P}}-a.s.,
	$$
	sending $n\rightarrow\infty$, by the lower semicontinuity of $\psi$, we have
	$$
	\int_s^t\left(\rho_{u}-Y_{u}\right) d K_{u} \leq \int_s^t\left[\psi\left(\rho_{u}\right)-\psi\left(Y_{u}\right)\right] d u \qquad \overline{\mathbb{P}}-a.s.,
	$$
	and thus
	$$
	\int_s^t\left(\rho_{u}-X_{u}\right) d \phi_{u} \leq \int_s^t\left[\psi\left(\rho_{u}\right)-\psi\left(X_{u}\right)\right] d u \qquad a.s..
	$$
	\vspace{0.1cm}
	
	\noindent\textbf{Step (A.4).} In this step, we will show that along a subsequence $
	\sup _{t \leq T} W_1\left(\mu_t^{(n)}, \mu_t\right) \rightarrow 0$,
	where $\mu_t$ denotes the distribution of $X_t$ for $t \in[0, T]$.
	To this aim, we apply the weak convergence in $\mathcal{P}_1(\mathbb{R})$ in the Wasserstein sense, as described in Definition 	\ref{def:Wasserstein_cvg}.

	For every $t\in [0,T]$, by equations \eqref{eqn:mvsvi_delta} and \eqref{eqn:mvsvi_delta1},
	for any $\epsilon>0$, there exists $R_{\epsilon}>0$ such that
	$$
	\sup_n\mu_t^{(n)}\big(\left[-R_{\epsilon}, R_{\epsilon}\right]^c\big)=\sup_n\mathbb{P}\big(|X_t^{(n)}|>R_{\varepsilon}\big) \leq \frac{\sup_n\mathbb{E}|X_t^{(n)}|}{R_{\epsilon}}<\epsilon,
	$$
	which means that $\{\mu_t^{(n)}\}_n$ is tight and there exists a subsequence $\{\mu_t^{(n_k)}\}_{n_k}$ that converges weakly in the weak convergence topology.
	Note that by equation \eqref{eqn:mvsvi_delta}, H\"older's inequality, and Markov's inequality,
	\begin{align*}
		\int_{|x| \geq R}|x| d \mu_t^{(n)}(x)  =\mathbb{E}\left[|X_t^{(n)}|  \mathbbm{1}_{\{|X_t^{(n)}| \geq R\}}\right]  \leq \frac{1}{R} C_T\left(1+\mathbb{E}\left|X_0\right|^2\right)  \rightarrow 0,\qquad \text { as } R \rightarrow  \infty .
	\end{align*}
	Hence,
	$$\lim _{R \rightarrow \infty} \lim _{n \rightarrow \infty} \int_{|x| \geq R}|x| d \mu_t^{(n)}(x)=0,$$
	and it then follows from Definition 	\ref{def:Wasserstein_cvg} that $\{\mu_t^{(n)}\}_{n }$ is relatively compact in $\mathcal{P}_1(\mathbb{R})$, and the subsequence  $\{\mu_t^{(n_k)}\}_{n_k}$ converges in $\mathcal{P}_1(\mathbb{R})$ weakly.
	Furthermore, for any $0 \leq s \leq t \leq T$ and all $n \geq 1$,
	\begin{align*}
		W_1\left(\mu_t^{(n)}, \mu_s^{(n)}\right)&\leq \mathbb{E}|X_t^{(n)}-X_s^{(n)}|\mathbbm{1}_{\{\sup_{ t\leq T}|X_t^{(n)}| \leq R\}}+\mathbb{E}|X_t^{(n)}-X_s^{(n)}|\mathbbm{1}_{\{\sup_{ t\leq T}|X_t^{(n)}| > R\}}\\
		&=: \mathcal{I}_1+\mathcal{I}_2 .
	\end{align*}
	By equation \eqref{eqn:mvsvi_delta1},
	$$\mathcal{I}_2\leq \frac{C_T}{R}(1+\mathbb{E}|X_0|^2).$$
	According to arguments in Step (A.1),
	$$
	\mathcal{I}_1 \leq C\left(R+\left|a_0\right|+1\right)^{l+2}(t-s) +C \mathbb{E}|\phi^{(n)}|_0^T\left[(t-s)+g_R(t-s) \right]+CR(t-s)^{1 / 2}.
	$$
	Summing up the above results, we obtain by equation \eqref{eqn:mvsvi_delta1}  again that
	\begin{align*}
		\lim _{\xi \rightarrow 0} \sup _n \sup _{|t-s| \leq r} W_1\left(\mu_t^{(n)}, \mu_s^{(n)}\right)&\leq \frac{C_T}{R}\left[1+ \mathbb{E}\left|X_0\right|^2\right]+\lim_{r\rightarrow 0}C_{\left(R, T\right)} \left(\xi^{1/2}+g_R(\xi)\right)\\ &=\frac{C_T\left(1+\mathbb{E}\left|X_0\right|^2\right)}{R} \longrightarrow 0, \quad\text { as } R \rightarrow \infty.
	\end{align*}
	Hence, $t\rightarrow \mu_t^{(n)}$ is continuous w.r.t. $W_1$
	and by Ascoli-Arzela's theroem, $\left\{\mu^{(n)}\right\}_{n}$ has compact closure in $C\left([0, T] ; \mathcal{P}_1(\mathbb{R})\right)$. Hence, there exists a subsequence $\{n_k\} \subset \{n\}$ and some $\widehat{\mu} \in C\left([0, T] ; \mathcal{P}_1(\mathbb{R})\right)$ such that
	\begin{align}
		\label{eqn:threehash}
		\lim _{n_k \rightarrow \infty} \sup _{t\leq T} W_1\left(\mu_t^{(n_k)}, \widehat{\mu}_t\right)=0,
	\end{align}
	where $\mu_t^{\left(n_k\right)}$ is the distribution of $X_t^{(n_k)}$.
	It remains to show that $\widehat{\mu}_t=\mu_t$ for all $t\in[0, T]$.
	By Definition \ref{def:Wasserstein_cvg} and Theorem \ref{thm:Wasserstein_cvg}, for all $t \in[0, T]$, $\mu_t^{\left(n_k\right)}$ converges weakly to $\widehat{\mu}_t$.
	Noting that $\left\{X^{\left(n_k\right)}\right\}_{n_k}$ converges in distribution to $X$,
	it then follows $\mu_t=\widehat{\mu}_t$.
	Define for $\left(x_1, x_2, \mu\right), (y_1, y_2, \nu)\in \mathbb{R}^2\times\mathcal{P}_1(\mathbb{R})$,
	$$d\big(\left(x_1, x_2, \mu\right),\left(y_1, y_2, \nu\right)\big):=\left|x_1-y_1\right|+\left|x_2-y_2\right|+W_1(\mu, v).$$
	Then it follows from equation \eqref{eqn:threehash} and \eqref{YKXP} that along a subsequence of $\{n\}$,
	\begin{align}
		\label{eqn:fourhash}
		\sup _{t \leq T} d\left((Y^{(n+1)}_t, Y^{(n)}_t, \mu^{(n)}_t),(Y_t, Y_t, \mu_t)\right) \longrightarrow 0 \quad\overline{\mathbb{P}}-\text {a.s.}.
	\end{align}
	Hence, $(X^{(n+1)}, X^{(n)}, \mu^{(n)}) \longrightarrow(X, X, \mu)$  in distribution in $C\big([0, T] ; \mathbb{R}^2 \times \mathcal{P}_1(\mathbb{R})\big).$
	
	By the conditions on $b$ and $\sigma$, and by equation \eqref{eqn:fourhash},
	\begin{align*}
		&\overline{\mathbb{E}} \int_0^T\left|b\left(s, Y_s^{(n+1)}, \mu_s^{(n)}\right)-b\left(s, Y_s, \mu_s\right)\right| d s\\
		& \leq \overline{\mathbb{E}} \int_0^T\Big(1+|Y_s^{(n+1)}|+\overline{\mathbb{E}}|Y_s^{(n)}|^{p_0}+|Y_s|+\overline{\mathbb{E}}|Y_s|^{p_0}\Big) \Big(|Y_s^{(n+1)}-Y_s^{(n)}|+W_1(\mu_s^{(n)}, \mu_s)\Big) d s \\
		& \leq C_{T,p_0}\overline{\mathbb{E}} \int_0^T\left(1+2 R\right)  |Y_s ^{(n+1)}-Y_s|d s\\
		& \quad+C_{T,p_0}\overline{\mathbb{E}} \int_0^T \left(1+|Y_s^{(n+1)}|+|Y_s|\right)| Y_s^{(n+1)}-Y_s^{(n)}| \left[\mathbbm{1}_{\{|Y_s^{(n+1)}|> R\}}+\mathbbm{1}_{\{|Y_s^{(n)}|> R\}}\right] d s\\
		& \quad+C_{T,p_0}\int_0^T \left(1+\mathbb{E} |Y^{(n+1)}|+\mathbb{E}|Y_s|\right) W_1(\mu_s^{(n)}, \mu_s) ds,
	\end{align*}
	which goes to $0$ as $n \rightarrow \infty$ and then $R \rightarrow \infty$. Furthermore,
	\begin{align*}
		&\hspace{-1cm}\overline{\mathbb{E}}  \int_0^T\left|\sigma(s, Y_s^{(n+1)}, \mu_s^{(n)})-\sigma\left(s, Y_s, \mu_s\right)\right|^2 d s \\
		&\leq C \overline{\mathbb{E}}  \int_0^T\left(1+|Y_s^{(n+1)}|+|Y_s|+\overline{\mathbb{E}}|Y_s^{(n)}|^{p_0}+\overline{\mathbb{E}}|Y_s|^{p_0}\right)\\
		&\hspace{1.7cm}\times |Y_s^{(n+1)}-Y_s|\left[|Y_s^{(n+1)}-Y_s|^{2 \alpha}+W_1(\mu_s^{(n)}, \mu_s)\right] d s,
	\end{align*}
	which goes to $0$ as $n\to\infty$.
	Hence, we have that in distrbution
	$$\int_0^{\cdot} b\left(s, X_s^{(n+1)}, \mu_s^{(n)}\right) ds \longrightarrow\int_0^{\cdot} b\left(s, X_s, \mu_s\right) d s,$$
	$$
	\int_0^{\cdot} \sigma\left(s, X_s^{(n+1)}, \mu_s^{(n)}\right) d B_s \longrightarrow \int_0^{\cdot} \sigma\left(s, X_s, \mu_s\right) d B_s.
	$$
	Then
	$$X_t=X_0+\int_0^t b\left(s, X_s, \mu_s\right) d s+\int_0^t \sigma\left(s, X_s, \mu_s\right) d B_s-\phi_t,$$
	and for any $0\leq s\leq t\leq T$ and any $\rho \in C([0,T],  \overline{D(\partial\psi)})$, we have $|\phi|_0^T<+\infty$ a.s. and
	$$
	\int_s^t\left(\rho_{u}-X_{u}\right) d \phi_{u} \leq \int_s^t\left[\psi\left(\rho_{u}\right)-\psi\left(X_{u}\right)\right] d u \quad \text { a.s.. }
	$$
	Thus, $(X, \phi)$ is a solution of the MVSVI \eqref{eqn:mvsvi}.
	\bigskip
	
	\noindent\textbf{Step (B) Uniqueness}
	\smallskip
	
	\noindent Now, we assume that $(\widehat{X}, \widehat{\phi})$ is also a solution. Then by Proposition \ref{thm:mvsvi_prop}, for all $t \in[0,T]$, there exists $0< \widehat{a}_t<a_0$ such that
	$
	\sup _{0 \leq s \leq t} \mathbb{E}e^{\widehat{a}_t|\widehat{X}_s|}<+\infty
	$,
	and for all $p\geq 1$,
	$$
	\mathbb{E}\sup _{t\leq T} |\widehat{X}_t|^p+\mathbb{E}\left(|\widehat{\phi}|_0^T\right)^{\frac{p}{2}} \leq C_{p,T}\left(1+\mathbb{E}|X_0|^p\right)<\infty.
	$$
	Similar to Step (A.2), applying It\^o's formula to $V_{\epsilon, \delta}(X_t-\widehat{X}_t)$ yields
	\begin{align*}
		|X_t-\widehat{X}_t|  \leq V_{\epsilon,\delta}(X_t-\widehat{X}_t)+\epsilon &=\int_0^t V_{\epsilon, \delta}^{\prime}(X_s-\widehat{X}_s)\left[b\left(s, X_s, \mu_s\right)-b(s, \widehat{X}_s, \widehat{\mu}_s)\right] d s \\
		& \quad+\frac{1}{2} \int_0^t V_{\epsilon, \delta}^{\prime \prime}(X_s-\widehat{X}_s) \left[\sigma\left(s, X_s, \mu_s\right)-\sigma(s, \widehat{X}_s, \widehat{\mu}_s)\right]^2 d s\\
		& \quad+\int_0^t V_{\epsilon, \delta}^{\prime}(X_s-\widehat{X}_s)\left[\sigma\left(s, X_s, \mu_s\right)-\sigma(s, \widehat{X}_s, \widehat{\mu}_s)\right] d B_s\\
		& \quad-\int_0^t V_{\epsilon, \delta}^{\prime}(X_s-\widehat{X}_s) d(\phi_s-\widehat{\phi}_s)+\epsilon\\
		&=:\sum_{i=1}^4 \mathcal{I}_{i(t)}+\epsilon,
	\end{align*}
	where $\mathcal{I}_4(t) \leq 0$ and $\mathbb{E}\, \mathcal{I}_3(t)=0$.
	By Proposition \ref{thm:mvsvi_prop} and Markov's inequality,
	\begin{align*}
		&\mathbb{E}\, \mathcal{I}_1(t)\\
		&\leq	 C \mathbb{E}\int_0^t \left(1+\left|X_s\right|+|\widehat{X}_s|+\mathbb{E}\left|X_s\right|^{p_0}+\mathbb{E}|\widehat{X}_s|^{p_0}\right) |X_s-\widehat{X}_s| d s \\
		& \quad+C \int_0^t\mathbb{E}\left(1+\left|X_s\right|+|\widehat{X}_s|+\mathbb{E}\left|X_s\right|^{p_0}+\mathbb{E}|\widehat{X}_s|^{p_0}\right)\mathbb{E}|X_s-\widehat{X}_s|ds\\
		& \leq C_{p_0,T}\int_0 ^t\left(1+2 R_s\right) \mathbb{E}|X_s-\widehat{X}_s| d s+C_{p_0,T} \int_0^t \mathbb{E}|X_s-\widehat{X}_s| d s\\
		& \quad+C\mathbb{E} \int_0^t\left(1+\left|X_s\right|+|\widehat{X}_s|+\mathbb{E}\left|X_s\right|^{p_0}+\mathbb{E}|\widehat{X}_s|^{p_0}\right)|X_s-\widehat{X}_s|  \Big(\mathbbm{1}_{\{ X_s>R_s \}}+\mathbbm{1}_{\{ \widehat{X}_s>R_s \}}\Big) d s\\
		& \leq C_{p_0, T} \int_0^t\left(1+2 R_s\right) \mathbb{E}|X_s-\widehat{X}_s| d s+C_{p_0, T} \int_0^t\left[\sqrt{\mathbb{P}( |X_s|>R_s)}+\sqrt{\mathbb{P}(|\widehat{X}_s|>R_s)}\right] d s \\
		& \leq C_{p_0, T} \int_0^t\left(1+2 R_s\right) \mathbb{E}|X_s-\widehat{X}_s| d s+C_{p_0, T} \int_0^t e^{-\frac{a_T R_s}{2}} d s,
	\end{align*}
	and
	\begin{align*}
		\mathbb{E}\, \mathcal{I}_2(t)&\leq	 \frac{1}{2\ln \delta} \mathbb{E}\int_0^t \left(1+\left|X_s\right|+|\widehat{X}_s|+\mathbb{E}\left|X_s\right|^{p_0}+\mathbb{E}|\widehat{X}_s|^{p_0}\right)\\
		&\hspace{4cm}\times\mathbbm{1}_{\{|X_s-\widehat{X}_s|\in[\frac{\epsilon}{\delta},\epsilon]\}} \big(\epsilon^{2\alpha}+|X_s-\widehat{X}_s|\big) d s \\
		&\leq C_{\delta, T}\epsilon^{2\alpha}+\frac{1}{2\ln \delta}\int_0^t\mathbb{E}\,|X_s-\widehat{X}_s|ds.
	\end{align*}
	Define an auxiliary quantity
	\begin{align}
		\label{eqn:def_Lambda}
		\Lambda{(t)}:=\sup _{t\leq T} \mathbb{E}|X_t-\widehat{X}_t|  \Big(\mathbb{E}\sup_{t \leq T}|X_t|+\mathbb{E}\sup_{t \leq T}|\widehat{X}_t|\Big)^{-1} e^{-2}.
	\end{align}
	Letting $\epsilon\rightarrow 0$ yields $\Lambda(t) \leq e^{-2}$ for all $t \in[0, T]$ and
	\begin{align*}
		\Lambda(t) \leq C_{p_0, T} \int_0^t\left(1+2 R_s\right) \Lambda(s) d s+C_{p_0, T} \int_0^t e^{-\frac{a_T R_s}{2}} d s.
	\end{align*}
	
	Let
	$t_0:=\inf \big\{t>0 ; \;\Lambda(t)>0\big\}$
	and assume $t_0<T$.
	Then for all $t \in[0, t_0)$, we have $\mathbb{E}|X_t-\widehat{X}_t|=0$. For $t \in(t_0, T]$,
	\begin{align*}
		0<\Lambda(t) &\leq C_{p_0, T} \int_{t_0}^t\left(1+\frac{a_T}{2}R_s\right) \Lambda(s) d s+C_{p_0, T} \int_0^t e^{-\frac{a_T}{2}R_s} d s.
	\end{align*}
	Take $R_s=\frac{-2\ln \Lambda(s)}{a_T}$. Then we have for  any $h \in(0, e^{-2})$,
	\begin{align*}
		\Lambda(t) \leq C_{p_0, T} \int_{t_0}^t(2-\ln \Lambda(s)) \Lambda(s) d s &\leq-2 C_{p_0, T}  \int_{t_0}^t \Lambda(s) \ln \Lambda(s) d s\\
		&\leq h-2 C_{p_0, T} \int_{t_0}^t \Lambda(s) \ln \Lambda(s) d s.
	\end{align*}
	Set $\Gamma(t):=-\int_t^{e^{-2}}\frac{d s}{s\ln s}$, for any $t \in(0, e^{-2})$.
	By Osgood's lemma (see Theorem \ref{thm:Osgood}), we have
	$$
	\Gamma(h)-\Gamma(\Lambda(T)) \leq C_{p_0, T}  \int_{t_0}^T 2 d s=2 C_{p_0, T}\left(T-t_0\right) .
	$$
	Sending $h \rightarrow 0$ gives
	$$\Gamma(0) \leq 2 C_{p_0, T}\left(T-t_0\right)+\Gamma(\Lambda(T))<\infty,$$ which contradicts to the fact that $\Gamma(0)=\infty$.
	Therefore, $\sup_{0\leq t\leq T} \Lambda(t)=0$, i.e., for all $t \in [0,T]$,
	$$
	\mathbb{E}|X_t-\widehat{X}_t|=0,
	$$
	which means that for all $t \in [0, T]$, $\mathbb{P}(X_t=\widehat{X}_t)=1$.
	By the continuity of $X$ and $\widehat{X}$, we have $$\mathbb{P}\Big(X_t=\widehat{X}_t,\; t \in [0,T]\Big)=1,$$
	as desired.
\end{proof}

\section{Propagation of chaos}
\label{sec:POC}

We first establish the strong well-posedness of the $N$-particle system \eqref{eqn:eqn_poc} recalled here as follows: for $1\leq i\leq N$ with $N\in \mathbb{N}$,
\begin{align*}
	d X_{t}^{N, i} \in b(X_{t}^{N, i}, \mu_{t}^{N}) d t+\sigma(X_{t}^{N, i}, \mu_{t}^{N}) d B_{t}^{i}-\partial \psi(X_{t}^{N, i}) d t,
\end{align*}
where
$\{X_{0}^{N, i}\}_{1\leq i\leq N}$ are i.i.d., 
$\mu^{N}=\frac{1}{N} \sum_{i=1}^{N} \delta_{X^{N, i}},$
and $\{B^i\}_{1\leq i\leq N}$ are independent standard Brownian motions.
\begin{theorem}
	\label{thm:wellposedness_POC}
	Suppose Assumption \ref{assumption2} holds, and  $\{X_{0}^{N, i}; 1\leq i\leq N\}$ are i.i.d. random variables for every $N\in\mathbb{N}$ such that $\sup _{N \in \mathbb{N}} \mathbb{E}\,e^{a_{0}|X_{0}^{N,i}|} <+\infty$. Then for every $1\leq i\leq N$, equation \eqref{eqn:eqn_poc} has a unique strong solution $(X^{N, i},\phi^{N,i})$ satisfying that for any $t\in[0,T]$, there exists a function $0<a_t<a_0$ such that \begin{align*}
		\sup _{N}\sup _{s\in[0,t]} \mathbb{E} e^{a_{t}|X_{s}^{N,i}|}<+\infty \quad\text{and} \quad
		\sup _{N} \mathbb{E}\sup _{t \leq T}|X_{t}^{N,i}|^p<+\infty,\quad \forall p\geq1.
	\end{align*}
\end{theorem}

\begin{proof}
	Set $X_{t}^{N, i,(0)} \equiv X_{0}^{N, i}$ and consider
	\begin{align}\label{XNi}
		X_{t}^{N, i,(n)}:=&X_{0}^{N, i}+\int_{0}^{t} b(X_{s}^{N, i,(n)}, \mu_{s}^{N,(n-1)}) d s+\int_{0}^{t} \sigma(X_{s}^{N, i,(n)}, \mu_{s}^{N,(n-1)}) d B_{s}^i\nonumber\\
		&-\int_{0}^{t} \partial \psi(X_{s}^{N, i,(n)}) d s,
	\end{align}
	where $n \geq 1$ and 
	$$\mu^{N,(n-1)}:=\frac{1}{N} \sum_{i=1}^{N} \delta_{X^{N, i,(n-1)}}.$$
	Then for every $n\geq1$, assume $X^{N, i,(n-1)}$ is well-defined and $\mathbb{E}\sup_{t\leq T}|X^{N, i,(n-1)}|^p<\infty$ for any $p\geq1$. By arguments analogous to the proof of Theorem \ref{thm:svi_wellposedness} and that of Proposition \ref{thm:mvsvi_prop}, a unique strong solution $(X^{N, i,(n)}, \phi^{N, i,(n)})$ exists for equation \eqref{XNi}, $\Big(X^{N,i,(n)}, \phi^{N,i,(n)}, X^{N,i,(n-1)}, \phi^{N,i,(n-1)}\Big)_{n}$ is tight, and there exists some $0<a_{t}<a_{0}$ such that
	\begin{align}
		\label{eqn:eqn_star}
		\sup _{s \in[0, t]}\mathbb{E} e^{a_{t}|X_{s}^{N, i,(n)}|}<+\infty \quad\text{and}\quad \sup_{n} \mathbb{E} \sup_{t \leq T}|X_{t}^{N, i,(n)}|^{p}<+\infty.
	\end{align}
	By equation \eqref{eqn:eqn_star}, for every $t\in[0,T]$,
	\begin{align}
		\label{eqn:eqn_d_star}
		\mathbb{P}\Big(|X_{t}^{N, i,(n)}|>R_{t}^{(n)}\Big) \leq C_{T} e^{-a_{T} R_{t}^{(n)}}.
	\end{align}
	Note that we have
	\begin{align*}
		W_1\left(\mu_{s}^{N,(n)}, \mu_{s}^{N,(n-1)}\right) &=W_{1}\left(\frac{1}{N} \sum_{i=1}^{N} \delta_{X_{s}^{N, i,(n)}}, \frac{1}{N} \sum_{i=1}^{N} \delta_{X_{s}^{N, i,(n-1)}}\right) \\
		&\leq \frac{1}{N} \sum_{i=1}^{N}\Big|X_{s}^{N, i,(n)}-X_{s}^{N, i,(n-1)}\Big|,
	\end{align*}
	where
	\begin{align*}
		\Big|X_{t}^{N, i,(n)}-X_{t}^{N,i,(n-1)}\Big| \leq V_{\epsilon, \delta}\Big(X_{t}^{N, i,(n)}-X_{t}^{N, i,(n-1)}\Big)+\epsilon =: \sum_{i=1}^{4} \mathcal{A}_{i}(t)+\epsilon.
	\end{align*}
	Here,
	\begin{align*}
		\mathcal{A}_1(t)&=\int_{0}^{t}V_{\epsilon,\delta}'\Big(X_{s}^{N, i,(n)}-X_{s}^{N, i,(n-1)}\Big)\Big[b(X_{s}^{N, i,(n)},\mu_{s}^{N, (n-1)})-b(X_{s}^{N, i,(n-1)},\mu_{s}^{N, (n-2)}) \Big]\\
		&\leq  C_{T,p_0}\int_{0}^{t}\Big(1+| X_{s}^{N, i,(n)}|+| X_{s}^{N, i,(n-1)}|+\mathbb{E}| X_{s}^{N,i,(n)}|^{p_{0}}+\mathbb{E}|X_{s}^{N, i,(n-1)}|^{p_{0}}\Big) \\
		& \hspace{2.5cm}\times\Big(|X_{s}^{N, i,(n)}-X_{s}^{N, i,(n-1)}|+W_{1}(\mu_{s}^{N,(n-1)},\, \mu_{s}^{N,(n-2)})\Big) d s \\
		&	\leq	 C_{T,p_0} \int_{0}^{t}\Big(1+| X_{s}^{N, i,(n)}|+| X_{s}^{N, i,(n-1)}|+\mathbb{E}| X_{s}^{N,i,(n)}|^{p_{0}}+\mathbb{E}|X_{s}^{N, i,(n-1)}|^{p_{0}}\Big) \\
		& \hspace{2.5cm}\times\Big(|X_{s}^{N, i,(n)}-X_{s}^{N, i,(n-1)}|+\frac{1}{N} \sum_{i=1}^{N}|X_{s}^{N, i,(n-1)}-X_s^{N, i,(n-2)}|\Big) d s\\
		& \leq C_{T,p_0} \int_{0}^{t}(1+2 R_{s}^{(n)}) \Big|X_{s}^{N, i,(n)}-X_{s}^{N, i,(n-1)}\Big| d s \\
		&\quad +C_{T,p_0} \int_{0}^{t}\Big(1+|X_{s}^{N, i,(n)}|+|X_{s}^{N, i,(n-1)}|\Big) \Big|X_{s}^{N, i,(n)}-X_{s}^{N, i,(n-1)}\Big| \\
		& \hspace{5cm}\times\Big(\mathbbm{1}_{\{|X_{s}^{N, i, (n-1)}|>R_{s}^{(n)}\}}+\mathbbm{1}_{\{|X_{s}^{N, i, (n)}|>R_{s}^{(n)}\}}\Big) d s \\
		&\quad +C_{T,p_0} \int_{0}^{t}(1+2 R_{s}^{(n)}) \frac{1}{N} \sum_{i=1}^{N}\Big|X_{s}^{N, i,(n-1)}-X_{s}^{N, i,(n-2)}\Big| d s\\
		&\quad+C_{T,p_0} \int_{0}^{t}\Big(1+|X_{s}^{N, i,(n)}|+|X_{s}^{N, i,(n-1)}|\Big) \frac{1}{N} \sum_{i=1}^{N}\Big|X_{s}^{N, i,(n-1)}-X_{s}^{N, i,(n-2)}\Big|  \\
		& \hspace{5cm}\times\Big(\mathbbm{1}_{\{|X_{s}^{N, i, (n-1)}|>R_{s}^{(n)}\}}+\mathbbm{1}_{\{|X_{s}^{N, i, (n)}|>R_{s}^{(n)}\}}\Big) d s.
	\end{align*}
	Note that for every $n\geq1$, $X^{N, i,(n-1)}$ and $X^{N, j,(n-1)}$ have the same distribution. We have by equations \eqref{eqn:eqn_star} and \eqref{eqn:eqn_d_star} that
	\begin{align*}
		\mathbb{E} \mathcal{A}_{1}(t) \leq & C_{T,p_0} \int_{0}^{t}(1+2 R_{s}^{(n)}) \sup _{u \leq s} \mathbb{E}\Big|X_{u}^{N, i,(n)}-X_{u}^{N, i,(n-1)}\Big| d s +C_{T,p_0}\int_{0}^{t} e^{-\frac{a_{T} R_s^{(n)}}{2}} d s\\
		&+C_{T,p_0} \int_{0}^{t}(1+2 R_{s}^{(n)}) \mathbb{E}\Big|X_{s}^{N, i,(n-1)}-X_{s}^{N, i,(n-2)}\Big| d s.
	\end{align*}
	We have  $\mathbb{E}\mathcal{A}_2(t)=0$ where 
	\begin{align*}
		\mathcal{A}_2(t):=&\int_{0}^{t}V_{\epsilon,\delta}'\Big(X_{s}^{N, i,(n)}-X_{s}^{N, i,(n-1)}\Big)\Big[\sigma(X_{s}^{N, i,(n)},\mu_{s}^{N, (n-1)})\\
		&\hspace{6cm}-\sigma(X_{s}^{N, i,(n-1)},\mu_{s}^{N, (n-2)}) \Big]dB_s.
	\end{align*}
	Next, by Remark \ref{ros1}, and by arguments similar to Step (A.2) in the proof of Theorem \ref{thm:mvsvi_wellposedness},
	\begin{align*}
		\mathcal{A}_3(t)&=\int_{0}^{t}V_{\epsilon,\delta}'\Big(X_{s}^{N, i,(n)}-X_{s}^{N, i,(n-1)}\Big)d\left[\phi_{s}^{N, i,(n)}-\phi_{s}^{N, i,(n-1)}\right]\le0, \quad\mbox{a.s.};\\
		\mathcal{A}_4(t)&=\frac{1}{2}\int_{0}^{t}V_{\epsilon,\delta}''\big(X^{N,i,(n)}-X^{N,i,(n-1)}\big)\Big|  \sigma(X^{N,i,(n)},\mu_s^{N,(n-1)})\\
		&\hspace{6cm}-\sigma(X^{N,i,(n-1)},\mu_s^{N,(n-2)})\Big|^2ds\\
		& \leq C_{T} \epsilon^{2 \alpha} \int_{0}^{t}(1+2 R_{s}^{(n)}) d s+C_{T} \int_{0}^{t} e^{-\frac{a_{T}R_{s}^{(n)}}{2}} d s\\
		& \quad+C_{T} \int_{0}^{t}(1+2 R_{s}^{(n)}) \mathbb{E}\Big|X_{s}^{N, i,(n)}-X_{s}^{N,i,(n-1)}\Big| d s\\
		& \quad+C_{T} \int_{0}^{t}(1+2 R_{s}^{(n)}) \mathbb{E}\Big|X_{s}^{N, i,(n-1)}-X_{s}^{N, i,(n-2)}\Big| d s.
	\end{align*}
	Analogous to Step (A.2) in the proof of Theorem \ref{thm:mvsvi_wellposedness}, we obtain that as $n \rightarrow \infty $,
	\begin{align*}
		\mathbb{E}\sup_{ t\leq T}\Big|X_{t}^{N, i,(n)}-X_{t}^{N, i,(n-1)}\Big| \rightarrow 0\quad\text{and}\quad
		\mathbb{E}\sup_{ t\leq T}\Big|\phi_{t}^{N, i,(n)}-\phi_{t}^{N, i,(n-1)}\Big| \rightarrow 0.
	\end{align*}
	Thus there exists a process $X^{N, i}$ such that as $n \rightarrow \infty $,
	$$\mathbb{E}\sup_{ t\leq T}\Big|X_{t}^{N, i,(n)}-X_{t}^{N,{i}}\Big| \rightarrow 0.$$
	Hence, by subtracting a subsequence of $\Big\{X^{N, i,(n)}, \phi^{N, i,(n)}, X^{N,i,(n-1)},\phi^{N,i,(n-1)}\Big\}_n$ if necessary,
	\begin{align*}
		&\sup_{ t\leq T}\Big|X_{t}^{N, i,(n)}-X_{t}^{N, i}\Big| \stackrel{\text { a.s. }}{\longrightarrow} 0, \qquad \sup_{ t\leq T}\Big|\phi_{t}^{N, i,(n)}-\phi_{t}^{N, i}\Big| \stackrel{\text { a.s. }}{\longrightarrow} 0,\end{align*}
	and by Lemma \ref{thm:erny},
	\begin{align*}
		W_{1}\Big(\mu_{t}^{N,(n)},\, \mu_{t}^{N}\Big) \longrightarrow 0\qquad \text {where } \mu^{N}:=\frac{1}{N} \sum_{i=1}^{N} \delta_{X^{N, i}}.
	\end{align*}
	It then follows from arguments analogous to Step (A.5) in the proof of Theorem \ref{thm:mvsvi_wellposedness} that  $(X^{N, i}, \phi^{N, i})$ is a solution of equation \eqref{eqn:eqn_poc}, and uniqueness can be proved similarly to Step (B) in the proof of Theorem \ref{thm:mvsvi_wellposedness}.
	
	For the solution $(X^{N, i}, \phi^{N,i})$ of equation \eqref{eqn:eqn_poc}, we have
	\begin{align*}
		\sup _{t\leq T} \mathbb{E} e^{a_{T}|X_{t}^{N,i}|}<+\infty \quad\text{and}\quad
		\sup _{N} \mathbb{E}\sup _{t \leq T}|X_{t}^{N,i}|<+\infty,
	\end{align*}
	whose proof is analogous to that of Proposition \ref{thm:mvsvi_prop}.
\end{proof}

The POC is stated in the following theorem.
\begin{theorem}
	\label{thm:POC}
	Assume the conditions of Theorem \ref{thm:wellposedness_POC} hold,  $(\overline{X}_{0}^{i})_{i \geq 1}$ are i.i.d. random variables such that $\mathbb{E} e^{a_0|\overline{X}_{0}^{i}|}<+\infty$, and $\mathbb{E}|X_0^{N,i}-\overline{X}_0^i|\to0$ as $N\to\infty$. Then
	$$\mathbb{E}\sup _{t \leq T}\big|X_{t}^{N, i}-\overline{X}_{t}^{i}\big| \rightarrow 0,$$
	where $\overline{X}_{t}^{i}$ is defined in \eqref{eqn:MV_limit}, that is, 
	\begin{align*}
		d \overline{X}_{t}^{i} \in b(\overline{X}_{t}^{i}, \overline{\mu}_{t}) d t+\sigma(\overline{X}_{t}^{i}, \overline{\mu}_{t}) d B_{t}^{i}-\partial \psi(\overline{X}_{t}^{i}) d t,
	\end{align*}
	where $ \overline{\mu}_{t}$ denotes the distribution of $\overline{X}_{t}^{i}$.
\end{theorem}

\begin{proof}
	By Theorem \ref{thm:mvsvi_wellposedness}, for every $i\geq1$, equation \eqref{eqn:MV_limit} has a unique strong solution $(\overline{X}^{i},\overline{\phi}^{i})$ satisfying that there exists some $0<a_t<a_0$ such that
	\begin{align}\label{barXmom}
		\sup _{t\leq T} \mathbb{E}\, e^{a_{T}|\overline{X}_{t}^{i}|}<+\infty \quad\text{and} \quad
		\mathbb{E}\sup _{t \leq T}|\overline{X}_{t}^{i}|^p<+\infty, \quad \forall p\geq1.
	\end{align}
	Recall that ${\mu}^{N}=\frac{1}{N} \sum_{i=1}^{N} \delta_{{X}^{N,i}}$ and let $\overline{\mu}^{N}:=\frac{1}{N} \sum_{i=1}^{N} \delta_{\overline{X}^{i}}$. For $\overline{\mu}_t$ being the law of $\overline{X}_t$, we have
	\begin{align*}
		W_{1}(\mu_{t}^{N}, \overline{\mu}_{t}) \leq W_{1}(\mu_{t}^{N}, \overline{\mu}_{t}^{N})+W_{1}(\overline{\mu}_{t}^{N}, \overline{\mu}_{t})\leq \frac{1}{N} \sum_{i=1}^{N}|X_{t}^{N, i}-\overline{X}_{t}^{i}|+W_{2}(\overline{\mu}_{t}^{N}, \overline{\mu}_{t}).
	\end{align*}
	By Theorem 1 of \cite{fournier2015rate},
	\begin{align}\label{W2}
		\mathbb{E}W_{2}^2(\overline{\mu}_{t}^{N}, \overline{\mu}_{t})
		\leq  \left(N^{-\frac12}+N^{(2-q)/q}\right)\left(\mathbb{E} e^{a_t |\overline{X}_{t}|}\right)^{2/q}
		\leq C_TN^{-\frac12},
	\end{align}
	where $q>4$. It then follows that
	\begin{align}\label{EW1}
		\mathbb{E} W_{1}(\mu_{t}^{N}, \overline{\mu}_{t})
		\leq \mathbb{E}|X_{t}^{N, i}-\overline{X}_{t}^{i}|+\mathbb{E}W_{2}(\overline{\mu}_{t}^{N}, \overline{\mu}_{t})\leq \mathbb{E}|X_{t}^{N, i}-\overline{X}_{t}^{i}|+C_TN^{-\frac14}.
	\end{align}
	
	Applying the Yamada-Watanabe function and It\^o's formula, we have
	\begin{align*}
		|X_{t}^{N, i}-\overline{X}_{t}^{i}| \leq &
		V_{\epsilon, \delta}\Big(X_{t}^{N, i}-\overline{X}_{t}^{i}\Big)+\epsilon \\
		=&\int_0^tV'_{\epsilon, \delta}\Big(X_{s}^{N, i}-\overline{X}_{s}^{i}\Big)\left[b(X_{s}^{N, i},\mu^N_s)-b(\overline{X}_{s}^{i},\overline{\mu}_s)\right]ds\\
		&+\int_0^tV'_{\epsilon, \delta}\Big(X_{s}^{N, i}-\overline{X}_{s}^{i}\Big)\left[\sigma(X_{s}^{N, i},\mu^N_s)-\sigma(\overline{X}_{s}^{i},\overline{\mu}_s)\right]dB^i_s\\
		&+\frac12\int_0^tV''_{\epsilon, \delta}\Big(X_{s}^{N, i}-\overline{X}_{s}^{i}\Big)\left|\sigma(X_{s}^{N, i},\mu^N_s)-\sigma(\overline{X}_{s}^{i},\overline{\mu}_s)\right|^2ds\\
		&-\int_0^tV'_{\epsilon, \delta}\Big(X_{s}^{N, i}-\overline{X}_{s}^{i}\Big)
		\frac{X_{s}^{N, i}-\overline{X}_{s}^{i}}{|X_{s}^{N, i}-\overline{X}_{s}^{i}|}d(\phi^{N,i}_s-\overline{\phi}_s^{i})+\epsilon .\end{align*}
	By arguments similar to Step (B) in the proof of Theorem \ref{thm:mvsvi_wellposedness}, and by equations \eqref{barXmom} and \eqref{EW1}, we have the following results:
	\begin{align*}
		&\bullet \mathbb{E}\left|\int_0^tV'_{\epsilon, \delta}\Big(X_{s}^{N, i}-\overline{X}_{s}^{i}\Big)\left[b(X_{s}^{N, i},\mu^N_s)-b(\overline{X}_{s}^{i},\overline{\mu}_s)\right]ds\right|\\
		&\quad\leq C_{p_0,T}\int_0^t(1+R_s)\mathbb{E}|X_{s}^{N, i}-\overline{X}_{s}^{i}|ds+C_{p_0,T}N^{-\frac14}\int_0^t(1+R_s)ds+C_T\int_0^t e^{-\frac12a_TR_s}ds,\\
		&\bullet -\int_0^tV'_{\epsilon, \delta}\Big(X_{s}^{N, i}-\overline{X}_{s}^{i}\Big)
		\frac{X_{s}^{N, i}-\overline{X}_{s}^{i}}{|X_{s}^{N, i}-\overline{X}_{s}^{i}|}d(\phi^{N,i}_s-\overline{\phi}_s^{i})\leq0,\\
		&\bullet \mathbb{E}\int_0^tV'_{\epsilon, \delta}\Big(X_{s}^{N, i}-\overline{X}_{s}^{i}\Big)\left[\sigma(X_{s}^{N, i},\mu^N_s)-\sigma(\overline{X}_{s}^{i},\overline{\mu}_s)\right]dB^i_s=0,\\
		&\bullet \frac12\mathbb{E}\int_0^tV''_{\epsilon, \delta}\Big(X_{s}^{N, i}-\overline{X}_{s}^{i}\Big)\left|\sigma(X_{s}^{N, i},\mu^N_s)-\sigma(\overline{X}_{s}^{i},\overline{\mu}_s)\right|^2ds\\
		&\quad\leq  C_{p_0,T}\int_0^t(1+R_s)\Big(\mathbb{E}|X_{s}^{N, i}-\overline{X}_{s}^{i}|^{2\alpha}+\mathbb{E}|X_{s}^{N, i}-\overline{X}_{s}^{i}|\Big)ds\\
		&\hspace{1cm}+C_{p_0,T}N^{-\frac14}\int_0^t(1+R_s)ds+C_{p_0,T,\alpha}\int_0^t e^{-\frac12a_TR_s}ds.
	\end{align*}
	Furthermore, by Theorem \ref{thm:wellposedness_POC} and equation \eqref{barXmom}, we have
	$\mathbb{E}|X_{t}^{N, i}-\overline{X}_{t}^{i}|\to0$ as $N\to\infty$, and 
	\begin{align*}
		\mathbb{E}\sup_{t\leq T}|X_{t}^{N, i}-\overline{X}_{t}^{i}|
		\leq&\mathbb{E}|X_{0}^{N,i}-X_{0}^{i}| +C_{p_0,T} \int_{0}^{T}(1+R_{s}) \mathbb{E}|X_{s}^{N, i}-\overline{X}_{s}^{i}| d s\\
		&+C_{T,p_0}\int_{0}^{T}e^{-\frac{a_T}{2}R_s}ds+C_{T,p_0} \mathbb{E} \int_{0}^{T}(1+|X_{s}^{N,i}|+|\overline{X}_{s}^{i}|) W_{1}(\overline{\mu}_{s}^{N}, \overline{\mu}_{s}) d s .
	\end{align*}
	Set $u^{N}(t):=\mathbb{E} \sup _{s \leq t}|X_{s}^{N,i}-\overline{X}_{s}^{i}|$ and  $R_{s}:=-\frac{2}{a_T}\ln G^{N}(s)$ where
	\begin{align*}
		G^N(t)=e^{-2}u^N(t)\Bigg/\Big(\sup_N\mathbb{E}\sup_{ t\leq T}|X_t^{N,i}|+\mathbb{E}\sup_{ t\leq T}|\overline{X}_t^i|\Big).
	\end{align*} 
	By Theorem \ref{thm:wellposedness_POC} and equations \eqref{barXmom} and \eqref{W2}, we have
	\begin{align*}
		u^{N}(t) &\leq 2 u^{N}(0)  +C_{T,p_0} \int_{0}^{t}\Big[e^{-\frac{a_T}{2}R_{s}}+(1+\frac{a_T}{2}R_{s}) u^{N}(s)\Big] ds \\
		&\quad+\mathbb{E} \int_{0}^{t}\Big(1+|X_{s}^{N, i}|+|\overline{X}_{s}^{i}|\Big) W_{1}(\overline{\mu}_{s}^{N}, \overline{\mu}_{s}) d s\\
		& \leq 2 u^{N}(0)+C_{T,p_0} \int_{0}^{t}\Big[e^{-\frac{a_T}{2}R_{s}}+(1+\frac{a_T}{2}R_{s}) u^{N}(s)\Big] d s+C_{T,p_0}N^{-\frac{1}{4}}.
	\end{align*}
	Then $$G^{N}(t) \leq 2 G^{N}(0)+C_{T,p_0} \int_{0}^{t}(2-\ln G^{N}{(s)}) G^{N}(s) d s+C_{T,p_0} N^{-\frac{1}{4}}.$$
	By Osgood's Lemma   (see Theorem \ref{thm:Osgood}), using arguments similar to Step (B) in the proof of Theorem \ref{thm:mvsvi_wellposedness}, we obtain that as $N \rightarrow \infty$,
	$$G^N(T)\leq \Big( 2G^N(0)+C_{T,p_0}N^{-1/4} \Big)^{e^{-C_{T,p_0}^2}}\rightarrow 0.$$
	Hence, we have proved that for all $i\geq1$,
	$$
	\lim_n\mathbb{E}\sup_{s\leq T}|X_{s}^{N,i}-\overline{X}_{s}^{i}|=0.
	$$
\end{proof}


\section*{Acknowledgements}
The authors would like to thank the editors and referees for their valuable comments and suggestions, which are very helpful for improving the paper. The research of Jing Wu (corresponding author) is supported by NSFC (No. 12471144). 

\section*{Declarations}
\textbf{Conflict of interest} The authors declare that they have no conflict of interest.\\

\noindent\textbf{Data availability } All authors wrote the main manuscript text. All authors reviewed the manuscript.\\
	
\noindent \textbf{Author contribution} No datasets were generated or analyzed during the current study.

\bibliography{Article.bib}

\end{document}